\documentclass{amsart}

\usepackage{amsmath}
\usepackage{amsfonts}
\usepackage{amssymb,amscd,amsthm}
\usepackage{mathrsfs}
\usepackage[bookmarksnumbered, hidelinks, linktocpage, plainpages]{hyperref}
\usepackage{indentfirst}
\usepackage{extarrows}
\usepackage{enumitem}
\usepackage{mathtools}
\usepackage{appendix}
\usepackage{cleveref}
\usepackage{bm}

\newtheorem{theorem}{Theorem}[section]
\newtheorem{lemma}[theorem]{Lemma}

\theoremstyle{definition}
\newtheorem{definition}[theorem]{Definition}
\newtheorem{notation}[theorem]{Notation}
\theoremstyle{remark}
\newtheorem{remark}[theorem]{Remark}

\numberwithin{equation}{section}

\newcommand{\R}{\mathbb{R}}

\newcommand{\p}{\partial}
\newcommand{\OO}{\mathbb O}

\newcommand{\Z}{\mathbb Z}
\newcommand{\f}{\bm{f}}
\newcommand{\g}{\bm{g}}
\newcommand{\e}{\bm{e}}
\newcommand{\E}{\bm{E}}

\begin{document}\title[Discrete octonionic Dirac operators]{Discrete Octonionic Dirac Operators on Bounded Lattices: Boundary Integral Theory and Scaling Limits}

\author{Guangbin Ren}
\email[G.~Ren]{rengb@ustc.edu.cn}
\address{School  of Mathematical   Sciences, University of Science and Technology of China, Hefei 230026, China}

\author{Xin Zhao}
\email[X.~Zhao]{zx130781@mail.ustc.edu.cn}
\address{Department of Mathematics, University of Science and Technology of China, Hefei 230026, China}

 \date{}

\thanks{The authors were supported by the National Natural Science Foundation of China (Grant No.~12571090).}


\begin{abstract}
We develop a boundary integral theory for discrete octonionic Dirac operators on bounded subsets of $\Z_h^8$. In the octonionic setting the usual passage from Stokes' theorem to a Cauchy--Pompeiu formula is obstructed by non-associativity. We resolve this difficulty by introducing a star product that builds the associator directly into the boundary kernel. This yields discrete Cauchy--Pompeiu and Cauchy formulas, discrete Cauchy--Bitsadze and Teodorescu transforms, and a discrete Sokhotski--Plemelj jump formula with the associated projection operators. We then compare the discrete and continuous fundamental solutions by using the asymptotic expansion theory of multiple Fourier transforms due to Shivakumar and Wong. Since the periodic Fourier symbol has a finite constellation of corner singularities rather than a single singularity, the classical single-singularity method does not apply. As an application we obtain quantitative kernel estimates and prove a scaling-limit characterization of octonionic regularity: on sufficiently regular bounded domains, a $C^3$ octonionic function is regular if and only if it is the scaling limit of discrete regular functions.
\end{abstract}

\maketitle

\section{Introduction}

Boundary integral formulas are the organizing principle of analytic function theory: they recover interior values from boundary data, generate singular integral projections, and provide the mechanism behind discrete-to-continuum limits. In the discrete complex setting this program goes back to Isaacs, Ferrand and Duffin and was reshaped by Mercat, Bobenko--Mercat--Suris, Chelkak--Smirnov and Skopenkov, who connected discrete analyticity with boundary value problems, conformal invariants and scaling limits \cite{Isaac,Ferrand,Du,Me,BMS,CS2011,CS2012,Smirnov,Sm,Sk}. In parallel, the Dirac-operator approach to Clifford and quaternionic analysis was developed by G\"urlebeck and Spr\"ossig, Cerejeiras, K\"ahler, Ku and Sommen, and, on lattices, by Ren and Zhu in the discrete quaternionic setting \cite{CSSS2004,GS2,CKS,RZ2017,RZ2022}. The octonionic case stands precisely at the meeting point of these two traditions.

It is also the first place where the classical mechanism genuinely fails. Baez emphasized the exceptional role of the octonions in algebra and geometry \cite{B2002}. In analysis, Li, Peng and Qian showed that even continuous octonionic Cauchy integrals on Lipschitz surfaces require delicate control of non-associativity \cite{LPT2008}. On a lattice the conflict is sharper: the usual route from Stokes' theorem to Cauchy--Pompeiu is no longer formal, because associator terms survive, and exactness depends on matching the boundary weights with the forward and backward differences that build the symmetric discrete Dirac operator. Discrete octonionic constructions also appear in work of Krausshar and Legatiuk \cite{KLL2021,KL}. To the best of our knowledge, however, the present paper gives the first construction of a discrete Cauchy--Pompeiu formula on bounded lattice domains that accounts for non-associativity through a star product, together with a scaling-limit characterization of octonionic regularity.

The present paper resolves both difficulties. Our first idea is algebraic: we introduce a star product that incorporates the associator directly into the boundary kernel, so that the boundary term can still be written in a workable octonionic form. Our second idea is geometric: the discrete normal enters through the two dual pairings
\[
 n_l^+\leftrightarrow \partial_l^{-,h},
 \qquad
 n_l^-\leftrightarrow \partial_l^{+,h},
\]
which is exactly the pairing required by the symmetric operator
\[
 \bm D^h=\sum_{l=0}^7 \e_l\,\frac{\partial_l^{+,h}+\partial_l^{-,h}}{2}.
\]
This restores exact discrete Stokes and Cauchy--Pompeiu formulas, and from them the discrete Cauchy and Teodorescu operators, the Sokhotski--Plemelj jump formula, and the Plemelj projections. The analytic part of the paper addresses a different obstruction: the Fourier symbol of the discrete kernel has finitely many corner singularities on the torus, so Thom\'ee's single-singularity method \cite{Th} does not apply. We instead use the asymptotic expansion theory of multiple Fourier transforms due to Shivakumar and Wong \cite{SW1979,Wo}.

Our first main theorem is an exact discrete Cauchy--Pompeiu representation.

\medskip\noindent\textbf{Main theorem A.}
\emph{Let $B\subset \Z_h^8$ be bounded and let $\f:\overline B\to\OO$. Then}
\[
\chi_B(y)\f(y)
=\int_{\partial B}\bm K^h(x,y)*\f(x)\,dS(x)
+\int_B \E^h(y-x)(\bm D^h\f)(x)\,dV^h(x).
\]
\emph{In particular, if $\f$ is discrete regular in $B$, then $\f$ is recovered exactly from its boundary values.} See \Cref{2023-09-11 19:15:35}. This yields the discrete Sokhotski--Plemelj formula and the Plemelj projections; see \Cref{2023-09-29 10:56:55,2023-09-29 17:27:26,2023-09-29 18:34:59}.

The bridge to the continuum is the sharp kernel estimate
\[
|\E^h(x)-\omega(h^{-1}x)\E(x)|\le \frac{Ch}{|x|^8+h^8},
\qquad x\in \Z_h^8,
\]
proved in \Cref{2023-09-30 19:16:04}. This estimate is the key analytic input in the scaling-limit argument.

Our second main theorem shows that octonionic regularity is exactly the continuum shadow of discrete regularity.

\medskip\noindent\textbf{Main theorem B.}
\emph{Let $B\subset\R^8$ be a bounded domain with $C^3$-boundary, and let $B^h\subset\Z_h^8$ converge to $B$ with exterior excess $O(h^2)$. Then}
    $$\f\in C^3(\overline B,\OO)\ \text{is regular on }B
\quad\Longleftrightarrow\quad\exists\,\f^h:\Z_h^8\to\OO\ \text{discrete regular on }B^h
\ \text{with}\ 
\max_{B^h\cap B}\|\f^h-\f\|\to 0.$$

\emph{Thus a $C^3$ octonionic function is regular exactly when it is the scaling limit of discrete regular functions.} See \Cref{2023-09-11 20:34:07,2023-09-11 20:33:47}. In this sense our results place octonionic regularity on the same discrete-to-continuum footing on which Chelkak--Smirnov and Skopenkov treated discrete holomorphicity, and Ren and Zhu treated the associative quaternionic case \cite{CS2011,Sk,RZ2017,RZ2022}.

The paper is organized as follows. Section~2 fixes the octonionic and lattice notation. Section~3 proves the dual discrete Stokes formulas and derives the Cauchy--Pompeiu formula together with the Cauchy and Teodorescu operators. Section~4 compares the discrete and continuous fundamental solutions and proves the kernel estimate above. Section~5 develops the discrete Sokhotski--Plemelj theory and proves the scaling-limit characterization of octonionic regularity.

\section{Foundations of Discrete Octonionic Analysis}

In this section we collect the basic notions and notation used throughout the paper.

\subsection{Octonion Algebra}
The octonion algebra $\OO$ is the largest normed division algebra; see, for example, \cite{B2002}. It is an eight-dimensional real vector space with standard basis
  \[
\e_0=1,\e_1,\cdots,\e_7
\]
and multiplication rules
 \[
\e_i\e_j=\begin{cases}
    \e_j,\quad &i=0,\\
    \e_i,\quad &j=0,\\
    -\delta_{i,j} +\varepsilon_{ijk}\e_k, &otherwise,
 \end{cases}
\]
where $\delta_{i,j}$ symbolizes the Kronecker delta, and $\varepsilon_{ijk}$ is an entirely antisymmetric tensor that equals 1 when the sequence \(ijk\) is one of:
\[
123, 145, 176, 246, 257, 347, 365.
\]
The octonions, by nature, are non-commutative and non-associative, as evidenced by the relation:
  \[
(\e_1\e_2)\e_4=-\e_1(\e_2\e_4).
\]
  Further, the relationship can be generalized as:
   \begin{equation}
    \label{2023-09-13 18:46:08}
    (\e_i\e_j)\e_k=\begin{cases}
        \e_i(\e_j\e_k),& k=0,i,j,i\oplus j,\\\\
        -\e_i(\e_j\e_k),& otherwise.
     \end{cases}
 \end{equation}
For distinct $i,j\in\{1,\dots,7\}$, the symbol $i\oplus j$ denotes the unique index $m\in\{1,\dots,7\}$ such that $\e_i\e_j=\pm \e_m$; equivalently, $|\varepsilon_{ijm}|=1$. In particular, $\oplus$ is not addition modulo $8$; see \cite{RZhao}.
Throughout the paper, octonionic conjugation is given by
\[
\overline{\e_0}=\e_0,
\qquad
\overline{\e_j}=-\e_j,\quad j=1,\dots,7.
\]

 \begin{definition}
   The associator for any elements $a,b,c \in \OO$ is defined as
\[
[a,b,c]:=(ab)c-a(bc).
\]
 \end{definition}

 \begin{lemma}[\cite{S1954,B2002}]
   The octonion algebra $\OO$ exhibits both alternativity and flexibility, implying that
\[
[a,a,b]=[b,a,a]=0
\]
and
     \[
[a,b,a]=0
\]
for any elements $a,b \in \OO$.
 \end{lemma}

The octonions also satisfy the Moufang identities:

 \begin{lemma}
    For any $a,b,c\in \OO,$ one has \begin{equation*}
       (aba)c= a(b(ac)), \quad  c(aba)=((ca)b)a, \quad  a(bc)a= (ab)(ca).
    \end{equation*}
 \end{lemma}

\subsection{Basic concepts in the discrete octonionic analysis}
We work on the lattice $\Z_h^8=(h\Z)^8$, $h>0$, viewed as a discretization of $\R^8$. We identify a lattice point with an octonion by writing
\[ \bm x = \sum_{i=0}^7 \e_i x_i \]
for $x=(x_0,\dots,x_7)\in\Z_h^8$, where $\{\e_0,\dots,\e_7\}$ is the standard basis of $\OO$ over $\R$.

 \begin{definition}
 	Given a subset \(B\subset \Z^8_h\), we write \(M(B,\OO)\) for the set of all functions from \(B\) to \(\OO\). For a function \(\f\in M(\Z^8_h,\OO)\), the discrete octonionic Dirac operator is defined by
 	\[
 		\bm D^h \f = \sum_{l=0}^7 \e_l \p^h_l \f,
 		\qquad
 		\p^h_l=\frac12(\p_l^{+,h}+\p_l^{-,h}).
 	\]
 	Here
 	\[
 		\p_l^{+,h}\f(x)=\frac{\f(x+he_l)-\f(x)}{h},
 		\qquad
 		\p_l^{-,h}\f(x)=\frac{\f(x)-\f(x-he_l)}{h}.
 	\]
 	If \(\f\) is defined only on a subset \(A\subset \Z_h^8\), these expressions are used only at those points \(x\in A\) for which all shifted lattice points appearing in the formula also belong to \(A\).
 \end{definition}
 We also consider different types of discrete boundaries, opting for a two-layered boundary in this context.
 \begin{definition}
 	For a subset \(B\) within $\Z^8_h$, the boundary \( \p B \) of \(B\) is defined as:
 	\[ \p B = \{ x \in \Z^8_h | N(x) \cap B \neq \emptyset, N(x) \cap (\Z^8_h \setminus B) \neq \emptyset \} \]
 	where \(N(x)\) denotes the neighborhood of \(x\):
 	\[ N(x) = \{ x, x \pm he_0, \cdots, x \pm he_7 \}. \]
 \end{definition}
 We can also define the discrete closure \( \overline{B} \) and the interior \( B^\circ \) of \(B\) as:
 \[ \overline{B} = B \cup \p B, \quad B^\circ = B \setminus \p B. \]
 In this study, we represent the set of functions \( \f : \overline B \rightarrow \OO \) by \( M(\overline{B}, \OO) \).
 \begin{definition}
 	A discrete function \( \f \in M(\overline{B}, \OO) \) is termed discrete octonionic regular in \(B\) if \( \bm D^h \f(x)=0 \) for every \(x\in B^\circ\).
 \end{definition}
 Additionally, we delve into the concepts of discrete boundary measure and discrete outward normal vector in relation to $\Z^8_h$.
 \begin{definition}\label{2023-09-06 21:09:12}
 	For a subset \(B\) in $\Z^8_h$, the discrete outward normal vector at a boundary point \(x\) of \( \p B \) is given by:
 	\[ n = (n_0^+, n_0^-, \cdots, n_7^+, n_7^-), \]
 	where:
 	\[ n_l^{\pm}(x) = \frac{-2\p_l^{\pm,h} \chi_B(x)}{\sqrt{\sum_{i=0}^7 [\left( \p_i^{+,h} \chi_B \right)^2 + \left( \p_i^{-,h} \chi_B \right)^2 ]}}. \]
 	The discrete octonionic outside normal vector \( \bm n \) on \( \p B \) is then defined as:
 	\[ \bm n = \sum_{l=0}^7 \frac{1}{2}(n_l^+ + n_l^-) \e_l. \]
 \end{definition}
 \begin{remark}
 	The functions \( n_l^{\pm} \) are defined on the boundary \( \p B \subset \Z^8_h \). By extending them by zero to the whole lattice \( \Z^8_h \), we may interpret \Cref{2023-09-13 18:19:13} pointwise. This pointwise formulation is convenient in the proofs below and avoids the distributional language used in \cite{RZ2017} for the quaternionic case.
 \end{remark}

 \section{Octonionic Analysis: The Discrete Stokes Theorem}
 The realm of octonionic analysis presents a unique version of the continuous Stokes theorem, which differs notably from the classical interpretation found in complex analysis. This octonionic variant introduces an added term linked to the associator, a concept detailed in \cite{LPT2008}.

 Various methodologies have been explored concerning the discrete integral theory across different contexts. For insights into the complex and quaternionic scenarios, we refer to  references \cite{CS2011}, \cite{CS2012}, \cite{GH2001}, and \cite{RZ2017}.

 \subsection{Multiplicative Discrete Derivative Formula}
 A significant challenge in the discrete Stokes theorem is the complex nature of the discrete derivative formula under multiplication, especially when contrasted against the continuous scenario. To navigate this complexity, we introduce specific shift operators to assimilate the extra terms.

 \begin{definition}
 	Given any point $y$ in $\Z^8_h$, the translation operator
 	\[
 		\tau_y:M(\Z_h^8,\OO)\longrightarrow M(\Z_h^8,\OO)
 	\]
 	is defined by
 	\[
 		(\tau_y\f)(x):=\f(x-y).
 	\]
 	When \(\f\) is defined only on a subset of \(\Z_h^8\), the notation \(\tau_y\f(x)\) is used only when the value \(\f(x-y)\) is well defined.
\end{definition}

\begin{lemma}[Leibniz rule I]
	Let $B$ be a subset in $\Z^8_h.$ Then for any $\f,\g\in M(B,\OO),$ we have
	
	\begin{equation}
		\label{2023-09-06 21:57:31}
		\begin{aligned}
			\p^{+,h}_l(\f \g)&=\p_l^{+,h}\f\g+\tau_{-he_l}\f\p_l^{+,h}\g,\\
			\p^{+,h}_l(\f \g)&= \p^{+,h}_l \f \tau_{-he_l}\g +\f \p_l^{+,h} \g,\\
			\p^{-,h}_l(\f \g)&   = \p_l^{-,h}\f \g+ \tau_{he_l}\f \p_l^{-,h}\g,\\
			\p^{-,h}_l(\f \g)&=\p_l^{-,h}\f \tau_{he_l}\g+\f\p_{l}^{-,h}\g.
		\end{aligned}
	\end{equation}
\end{lemma}

\begin{proof}
	For the first two equalities,
	by definition we have
	\begin{equation*}
		\begin{aligned}
			\p_l^{+,h}(\f\g)(x)&= \frac{1}{h}(\f(x+he_l)\g(x+he_l)-\f(x)\g(x))\\&= \frac{1}{h} (\f(x+he_l)\g(x+he_l)-\f(x+he_l)\g(x)+\f(x+he_l)\g(x)-\f(x)\g(x)) \\&= \tau_{-he_l}\f(x) \p^{+,h}_l\g(x) +\p_l^{+,h}\f(x)\g(x),
		\end{aligned}
	\end{equation*}
	and \begin{equation*}
		\begin{aligned}
			\p_l^{+,h}(\f\g)(x)&=\frac{1}{h}(\f(x+he_l)\g(x+he_l)-\f(x)\g(x))\\&=\frac{1}{h}(\f(x+he_l)\g(x+he_l)-\f(x)\g(x+he_l)+\f(x)\g(x+he_l)-\f(x)\g(x))\\&=\p_l^{+,h}\f(x)\g(x+he_l)+\f(x)\p_l^{+,h}\g(x).
		\end{aligned}
	\end{equation*}
	The proofs for the remaining two equalities follow a similar approach.
\end{proof}

To unify the variables of functions on the right side of \eqref{2023-09-06 21:57:31}, we have the following alternatives.
\begin{lemma}
	[Leibniz rule II]
	\begin{equation}
		\label{2023-09-07 17:40:22}
		\begin{aligned}
			\p^{+,h}_l(\f \g)&= \p^{+,h}_l \f \tau_{-he_l}\g +\f \p_l^{-,h} \tau_{-he_l}\g,\\
			\p^{+,h}_l(\f \g)&=\p_l^{-,h} \tau_{-he_l}\f\g+\tau_{-he_l}\f\p_l^{+,h}\g,\\
			\p^{-,h}_l(\f \g)&=\p_l^{-,h}\f \tau_{he_l}\g+\f\p_{l}^{+,h}\tau_{he_l}\g,     \\\p^{-,h}_l(\f \g)&   = \p_l^{+,h}\tau_{he_l}\f \g+ \tau_{he_l}\f \p_l^{-,h}\g.
		\end{aligned}
	\end{equation}
\end{lemma}
\begin{proof}
	This follows from the fact that
	\begin{equation}
		\begin{aligned}
			\p_l^{+,h}&=\p_l^{-,h}\tau_{-he_l},\\
			\p_l^{-,h}&=\p_l^{+,h}\tau_{he_l}.
		\end{aligned}
	\end{equation}
	
\end{proof}

\subsection{Discrete Stokes theorem}
 In this subsection, we will employ the discrete derivative formula in the context of multiplication to formulate a discrete Stokes theorem without the supplementary term. This is in alignment with the methodology delineated in \cite{RZ2017, RZ2022}. As highlighted in \cite{RZ2017, RZ2022}, for a bounded domain  $B\in \R^8$ with a $C^1$-boundary, a unique solution exists for the Gauss system represented as:
 \begin{equation}
 	\label{2023-09-08 09:37:46}
 	\begin{cases}
 		n_l\, dS =-\dfrac{\p}{\p x_l}\chi_B\, dV,\\[2mm]
 		\sum\limits_{l=0}^7 n_l^2=1 \qquad \text{on }\p B.
 	\end{cases}
 \end{equation}
 Here, $dS$ denotes a non-negative regular Borel measure, defined by the restriction of the Lebesgue measure on $\p B$.

 Subsequently, we derive a discrete counterpart of this theorem. Throughout this paper, $dV^h$ symbolizes the Haar measure on the group $\Z^8_h$. More specifically, for any function $\f: \Z^8_h\longrightarrow \OO$, we define
 \[
\int_{\Z^8_h} \f dV^h =\sum\limits_{x\in \Z^8_h} \f(x)h^8.
\]

\begin{definition}\label{2023-09-06 21:08:52}
	Given a set $B$ in $\Z^8_h$, the discrete boundary measure, $S_{\p B}$, on $\p B$ is defined by
    \[
S_{\p B}(U):=\sum\limits_{x\in U} s(x)
\] for any $U\subset   \p B$ for any subset $U$ of $\p B$. Here, the function $s: \p B\longrightarrow \R$ is defined as \[
s(x):=\frac{h^8}{2}\sqrt{\sum\limits_{l=0}^7 \left((\p_l^{+,h}\chi_B)^2+(\p_l^{-,h}\chi_B)^2\right)}.
\]
\end{definition}

\begin{lemma}\label{2023-09-13 18:19:13}
	Let $B$ be a set in $\Z^8_h$. Let $n=(n_0^+,n_0^-,\cdots,n_7^+,n_7^-)$ be the external normal vector on $\p B$ and let $S=S_{\p B}$ represent the boundary measure of $\p B$. Then, we have
    \begin{equation}
        \label{2023-09-06 19:05:40}
        \begin{cases}
            n_l^{\pm} dS = -\p_l^{\pm,h } \chi_B dV^h  ,\\\\
            \sum\limits_{l=0}^7\left((n_l^+)^2+(n_l^-)^2\right)= 4\chi_{\p B}.
        \end{cases}
    \end{equation}

\end{lemma}

\begin{proof}
	We commence by proving the first equation in \eqref{2023-09-06 19:05:40}.
	
	Suppose $x \notin \p B$. Then, the relation $x \in B$ is equivalent to $x \pm he_l \in B$. This yields
	\begin{equation*}
		\p_l^{\pm,h} \chi_B(x) = 0.
	\end{equation*}
	Given the zero-extension, this implies:
	\begin{equation*}
		dS(x) = 0
	\end{equation*}
	Consequently, we establish:
	\begin{equation*}
		n_l^{\pm} dS = -\p_l^{\pm,h } \chi_B dV^h \quad \text{on} \quad \Z_h^8\setminus \p B.
	\end{equation*}
	
	For every $x \in \p B$, the following holds:
	\begin{align*}
		n_l^{\pm}(x)dS(x) &= n_l^{\pm}(x)s(x) \\
		&= \frac{-2\p_l^{\pm,h} \chi_B(x)}{\sqrt{\sum_{i=0}^{7} \left( (\p_i^{+,h}\chi_B)^2 + (\p_i^{-,h}\chi_B)^2 \right)}} \cdot \frac{h^8}{2}\sqrt{\sum_{i=0}^{7} \left( (\p_i^{+,h}\chi_B)^2 + (\p_i^{-,h}\chi_B)^2 \right)} \\
		&= -\p_l^{\pm,h}\chi_B(x)h^8 \\
		&= -\p_l^{\pm,h}\chi_B(x) dV^h.
	\end{align*}
	
	To ascertain the second equation in \eqref{2023-09-06 19:05:40}, we employ the predefined notion that
	\begin{equation*}
		\sum_{l=0}^{7} \left( (n_l^+)^2 + (n_l^-)^2 \right) \bigg|_{x\in \p B} = 4.
	\end{equation*}
	Given that $n_l^\pm$ is null outside $\Z^8_h\setminus \p B$, it follows:
	\begin{equation*}
		\sum_{l=0}^{7} \left( (n_l^+)^2 + (n_l^-)^2 \right) \bigg|_{x\notin \p B} = 0.
	\end{equation*}
	
	This concludes our proof.
\end{proof}

\begin{theorem}[Divergence Principle]\label{2023-09-06 21:25:07}
Let $B$ be a bounded set in  $\Z^8_h$. Then for every   $l=0,\cdots,7$, we have
	\begin{equation}
		\label{2023-09-06 20:18:53}
		\int_{\p B} \f n_l^{\pm} dS =\int_B \p^{\mp,h}_l \f dV^h
	\end{equation}
  for any function $\f \in M(\overline B, \OO)$.
\end{theorem}

\begin{proof}
	Extending the function $\f$ to encompass the entire space $\Z^8_h$, the relationship \eqref{2023-09-06 19:05:40} yields
	\begin{equation}
		\label{2023-09-06 20:57:00}
		\int_{\p B} \f n_l^{\pm}dS = \int_{\Z^8_h} \f n_l^{\pm}dS_{\p B} = -\int_{\Z^8_h} \f \p_l^{\pm,h} \chi_B dV^h.
	\end{equation}
	From the relation
	\begin{equation}\label{2023-09-08 10:36:15}
		\begin{aligned}
			-\int_{\Z^8_h} \f \p_l^{+,h} \chi_B dV^h &= -\frac{1}{h}\int_{\Z^8_h} \f(x)(\chi_B(x+he_l)-\chi_B(x))dV^h\\
			&= \int_{\Z^8_h} \p_l^{-,h} \f(x) \chi_B(x) dV^h(x)\\
			&= \int_B \p_l^{-,h} \f dV^h
		\end{aligned}
	\end{equation}
	and a parallel derivation indicating:
	\begin{equation}\label{2023-09-08 10:36:21}
		-\int_{\Z^8_h} \f \p_l^{-,h} \chi_B dV^h = \int_B \p_l^{+,h} \f dV^h,
	\end{equation}
	integrating the results of \eqref{2023-09-06 20:57:00}, \eqref{2023-09-08 10:36:15}, and \eqref{2023-09-08 10:36:21} completes the proof.
\end{proof}

\begin{theorem}[Stokes Formula in Discrete Octonionic setting]
	For a bounded set $B$ in $\Z^8_h$, for each $l=0,\cdots,7$ and for any functions $\f,\g \in M(\overline B, \OO)$, the following relations hold:
	\begin{equation}
		\label{2023-09-06 19:08:17}
		\begin{aligned}
			\int_{\p B} \f n_l^+ \g \, dS &= \int_B \p_l^{-,h}\f \, \g + \tau_{he_l}\f \, \p_l^{-,h}\g \, dV^h,\\
			\int_{\p B} \f n_l^+ \g \, dS &= \int_B \p_l^{-,h}\f \, \tau_{he_l}\g + \f \, \p_l^{-,h}\g \, dV^h,\\
			\int_{\p B} \f n_l^- \g \, dS &= \int_B \p_l^{+,h}\f \, \g + \tau_{-he_l}\f \, \p_l^{+,h}\g \, dV^h,\\
			\int_{\p B} \f n_l^- \g \, dS &= \int_B \p_l^{+,h}\f \, \tau_{-he_l}\g + \f \, \p_l^{+,h}\g \, dV^h.
		\end{aligned}
	\end{equation}
\end{theorem}

\begin{proof}
	By Theorem \ref{2023-09-06 21:25:07}, we have
	\begin{equation*}
		\int_{\p B} \f n_l^+ \g \, dS = \int_B \p_l^{-,h}(\f\g) \, dV^h,
		\qquad
		\int_{\p B} \f n_l^- \g \, dS = \int_B \p_l^{+,h}(\f\g) \, dV^h.
	\end{equation*}
	The first two identities now follow from the last two formulas in \eqref{2023-09-06 21:57:31}, and the last two identities follow from the first two formulas in \eqref{2023-09-06 21:57:31}.
\end{proof}

\section[Cauchy-Pompeiu Formula on the 8D Lattice]{The Cauchy-Pompeiu Formula and Its Inverse on the Lattice \texorpdfstring{$\Z_h^8$}{Zh8}}

In this section, our focus lies on understanding the Cauchy-Pompeiu formula and its inverse in the realm of discrete octonionic analysis.

 For this purpose, we utilize the following relation:
\[
(\e_i\e_j)\e_k=\begin{cases}
	\e_i(\e_j\e_k),& \text{when } k=0,i,j,i\oplus j,\\\\
	-\e_i(\e_j\e_k),& \text{otherwise}.
\end{cases}
\]
This equation stems from \eqref{2023-09-13 18:46:08}. Here $i\oplus j$ has the same meaning as above: for distinct $i,j\in\{1,\dots,7\}$ it denotes the unique index $m$ for which $\e_i\e_j=\pm \e_m$.
  It is noteworthy that, in general,  
	\[
(\e_i\e_j)\e_k\neq -\e_i(\e_j\e_k).
\]

    \begin{definition}
        The discrete Laplacian $\Delta^h : M(\Z^8_h, \OO)\longrightarrow M(\Z^8_h, \OO)$ is defined by \[
\Delta^h  =  \sum\limits_{i=0}^7 \p_i^h\p_i^h .
\]
    \end{definition}

    For $h=1$,  we have \[
\Delta^1 \f (x)=\frac{1}{4}\sum\limits_{i=0}^7 (\f(x+2e_i)+\f(x-2e_i)-2\f(x))
\] for any $\f\in M(\Z^8_h,\OO).$
    The fundamental solution of $\Delta^1$ is given by   \begin{equation}
        \label{2023-09-07 20:07:19}
        F^1(x)=- \frac{1}{(2\pi)^8}\int_{[-\pi,\pi]^8}  {\left(\sum\limits_{l=0}^7 \sin^2u_l\right)^{-1}} e^{i\sum\limits_{l=0}^7 u_lx_l} du,
    \end{equation}
such that
\begin{equation}
	\label{2023-09-07 09:51:28}
	\Delta^1   F^1(x)=\delta^1_0(x)
\end{equation}
where
$\delta_0^1(x)$ is the discrete Dirac delta function defined by   \[
\delta_0^1(x)=\begin{cases}
	1,\quad x=0,\\\\
	0,\quad x\neq 0.
\end{cases}
\]
 Indeed, it can be shown that  \[
\Delta^1\left( e^{i\sum\limits_{l=0}^7 u_lx_l}\right)=\left(-\sum\limits_{l=0}^7 \sin^2 u_l\right) e^{i\sum\limits_{l=0}^7 u_lx_l},
\] which leads to
    \begin{equation}
        \label{2023-09-07 10:05:31}
        \begin{aligned}
            \Delta^1 F^1(x)&= -\frac{1}{(2\pi)^8} \int_{[-\pi,\pi]^8}\left(\sum\limits_{l=0}^7 \sin^2u_l\right)^{-1}  \left(-\sum\limits_{l=0}^7\sin^2 u_l\right) e^{i\sum\limits_{l=0}^7 u_lx_l} du\\&= \frac{1}{(2\pi)^8} \int_{[-\pi,\pi]^8} e^{i\sum\limits_{l=0}^7 u_lx_l} du \\&= \delta_0^1(x).
        \end{aligned}
    \end{equation}

    \begin{definition}
        A function $\E^h\in M(\Z^8_h,\OO)$ is called the fundamental solution of $\bm D^h$ if \begin{equation}
            \label{2023-09-15 10:18:26}
            \bm D^h \E^h =\delta_0^h (x)
        \end{equation}
        in a pointwise  sense,
         where $\delta_0^h (x)$  is the discrete Dirac delta function defined by   \[
\delta_0^h (x)=\begin{cases}
            \frac{1}{h^8},\quad x=0,\\\\
            0,\quad x\neq 0
        \end{cases}.
\]
    \end{definition}
   Notably,  \eqref{2023-09-15 10:18:26} holds in a pointwise sense, leading to \begin{equation}
        \label{2023-09-15 10:07:31}
        \delta^h_0(x)=\frac{1}{h^8}\chi_{\{0\}}(x).
    \end{equation}
This contrasts with the continuous case.

    \begin{lemma}\label{2023-09-07 22:03:59}
        The fundamental solution $\E^h(x)$ of $\bm D^h$ satisfies   \[
\E^h (x)= \frac{1}{h^7}\E^1(\frac{x}{h}),
\] where  \begin{equation}
            \label{2023-09-07 10:25:17}
            \E^1(x):= - \frac{1}{(2\pi)^8}\int_{[-\pi,\pi]^8} i\sum\limits_{l=0}^7  \bm {\overline   e}_l \sin (u_l)   {\left(\sum\limits_{l=0}^7 \sin^2(u_l)\right)^{-1}} e^{i\sum\limits_{l=0}^7 u_lx_l}  du.
        \end{equation}
    \end{lemma}
    \begin{proof}
     We may assume $h=1.$ Define the discrete octonionic Cauchy--Fueter operator
\[
\bm D^1 = \sum\limits_{i=0}^7 \e_i\p_i,
\]
and its conjugate operator
\[
\bm{\overline D}^{\,1}=\e_0\p_0-\sum\limits_{i=1}^7 \e_i\p_i.
\]
A direct calculation yields
\begin{equation}
        \label{2023-09-07 10:46:23}
        \Delta^1=\bm D^1 \bm{\overline D}^{\,1}.
       \end{equation}
With the definitions of $\E^1$ and $F^1$, and noting that
\[
\bm {\overline D}^{\,1} \Big(e^{i\sum\limits_{l=0}^7 u_lx_l}\Big)= e^{i\sum\limits_{l=0}^7 u_lx_l } \sum\limits_{l=0}^7 i \overline{\e_l} \sin (u_l),
\]
we deduce
\begin{equation}
        \label{2023-09-17 09:11:52}
        \bm {\overline D}^{\,1} F^1= \E^1.
       \end{equation}
From \eqref{2023-09-07 09:51:28} and \eqref{2023-09-07 10:46:23}, it follows that
\[
\bm D^1 \E^1=\bm D^1 \bm{\overline D}^{\,1} F^1=\Delta^1 F^1=\delta_0^1.
\]
This completes the proof.
    \end{proof}

    \begin{lemma}\label{2023-09-11 19:51:11}
    The fundamental solution $\E^1$ is an element of  $L^2(\Z^8).$
\end{lemma}
\begin{proof}
    Recall
    \[
\E^1(x):= - \frac{1}{(2\pi)^8}\int_{[-\pi,\pi]^8} i\sum\limits_{l=0}^7  \bm {\overline   e}_l \sin (u_l)   {\left(\sum\limits_{l=0}^7 \sin^2(u_l)\right)^{-1}} e^{i\sum\limits_{l=0}^7 u_lx_l}  du.
\]
It implies    $\E^1$ is the Fourier transform of
    \[
i\sum\limits_{l=0}^7  \bm {\overline   e}_l \sin (u_l)   {\left(\sum\limits_{l=0}^7 \sin^2(u_l)\right)^{-1}}.
\]
    We claim that this function is dominated by \[
\frac{1}{\sqrt{\sum\limits_{l=0}^7(u_l-s_l)^2}}
\]
    near any of its singular points $(s_0,s_1,\cdots,s_7).$
      Indeed, by periodicity, we only need to consider the singular point $(0,0,0,0,0,0,0,0)$, where
   \begin{equation*}
       \left \Vert i\sum\limits_{l=0}^7  \bm {\overline   e}_l \sin (u_l)   {\left(\sum\limits_{l=0}^7 \sin^2(u_l)\right)^{-1}} \right\Vert =\frac{1}{\sqrt{\sum\limits_{l=0}^7 \sin^2u_l}}   \thicksim \frac{1}{\sqrt{\sum\limits_{l=0}^7 u_l^2}}.
    \end{equation*}
   Thus it belongs to $L^2([-\pi,\pi]^8)$ and by Parseval's identity we have $\E^1 \in L^2(\Z^8).$
\end{proof}

\begin{remark}
   The integration expression of $\E^1 $ from  \eqref{2023-09-07 10:25:17} shows  that $x=0$ is not a singular point  of $\E^1$, contrasting the continuous case where
       \[
\E(x)=\frac{ 1}{\omega_7}\frac{\bm{\bar x}}{|\bm x|^8}.
\] Here  $\omega_7=\frac{2\pi^{4}}{\Gamma(4)}$ is the surface area of the unit sphere $S^7$.
\end{remark}

In the discrete context, the boundary integral kernel of $\bm D^h$
differs from its continuous counterpart. Specifically, while the continuous version is characterized by the fundamental solution $\E$, the discrete form is described by
 \begin{equation}
        \label{2023-09-07 16:01:11}
        \bm K^h(x,y):= -\frac{1}{2}\sum\limits_{l=0}^7 \left(\E^h(he_l-x+y)n_l^-(x)+\E^h(-he_l-x+y)n_l^+(x)\right)\e_l.
    \end{equation}
Given the octonionic nature of the terms within this kernel, multiplying $\bm K^h(x,y)$ by an octonionic boundary value, say \(\f(x)\), presents challenges because of non-associativity. For our purposes, we adopt the following ordered product:
  \begin{equation}
	\label{2023-09-07 21:56:06}
	\bm  K^h(x,y)*\f(x) := -\frac{1}{2}\sum\limits_{l=0}^7 \left(\E^h(he_l-x+y)n_l^-(x)+\E^h(-he_l-x+y)n_l^+(x)\right)(\e_l \f(x))
\end{equation}
 By using this approach, we can develop a Cauchy-Pompeiu formula for
$\bm D^h$
 that mirrors its continuous counterpart in form, but compensates with unique boundary and volume integral kernels. The pairing of $n_l^{-}$ with the shifted kernel $\E^h(he_l-x+y)$ and of $n_l^{+}$ with $\E^h(-he_l-x+y)$ is forced by \Cref{2023-09-06 20:18:53}: the boundary weights are dual to the forward and backward differences, and together they match the symmetric difference $\p_l^h$ exactly.

 \begin{theorem}[Cauchy-Pompeiu formula]\label{thm:cauchy-pompeiu}\label{2023-09-11 19:15:35}
 	Let $B$ be a bounded set in $\Z^8_h.$ Then for any function $\f:\overline B\longrightarrow  \OO,$ we have \begin{equation}
 		\begin{aligned}
 			\chi_B(y)\f(y)=\int_{\p B}\bm K^h(x,y) * \f(x)dS(x)+ \int_B \E^h(y-x)(\bm D^h \f)(x)dV^h(x)  .
 		\end{aligned}
 	\end{equation}
 \end{theorem}

         \begin{proof}
Employing the discrete divergence theorem on the boundary integral, we get
\begin{equation}\label{2023-09-08 09:07:56}
\begin{aligned}
 \int_{\p B}\bm K^h(x,y)*\f(x)dS(x)=  \int_B  G^h(x,y)\,dV^h(x),
\end{aligned}
\end{equation}
where
\[
G^h(x,y)= -\frac{1}{2} \sum\limits_{l=0}^7  \Bigg( \p_l^{+,h}\Big(\E^h(he_l-\cdot+y)(\e_l \f(\cdot))\Big)(x) +\p_l^{-,h}\Big(\E^h(-he_l-\cdot+y)(\e_l\f(\cdot))\Big)(x)\Bigg).
\]

For each fixed \(l\), the Leibniz rule and the fact that \( \p_l^{\pm,h}(\e_l\f)=\e_l\p_l^{\pm,h}\f \) give
\[
\begin{aligned}
&\p_l^{+,h}\Big(\E^h(he_l-\cdot+y)(\e_l \f(\cdot))\Big)(x)
+\p_l^{-,h}\Big(\E^h(-he_l-\cdot+y)(\e_l\f(\cdot))\Big)(x)\\
&\qquad=
\E^h(-x+y)\bigl(\e_l\p_l^{+,h}\f(x)+\e_l\p_l^{-,h}\f(x)\bigr)
-\bigl((\p_l^{+,h}\E^h)(-x+y)+(\p_l^{-,h}\E^h)(-x+y)\bigr)(\e_l\f(x))\\
&\qquad=
2\,\E^h(-x+y)(\e_l\p_l^h\f(x))-2(\p_l^h\E^h)(-x+y)(\e_l\f(x)).
\end{aligned}
\]
Hence
\[
G^h(x,y)
=- \E^h(-x+y)(\bm D^h \f)(x)+ \sum\limits_{l=0}^7 (\p_l^h\E^h)(-x+y)(\e_l \f(x)).
\]
Using the definition of the associator, we rewrite the second term as
\[
\sum_{l=0}^7 \bigl((\p_l^h\E^h)(-x+y)\e_l\bigr)\f(x)
-\sum_{l=0}^7[(\p_l^h\E^h)(-x+y),\e_l,\f(x)].
\]

We first show that
\begin{equation}
\label{eq:right-fundamental}
\sum_{l=0}^7 \bigl((\p_l^h\E^h)(z)\e_l\bigr)=\delta_0^h(z).
\end{equation}
Indeed, \(\E^h=\bm{\overline D}^{\,h}F^h\) with \(F^h\) real-valued, so
\[
(\p_l^h\E^h)(z)=\sum_{t=0}^7 \bm{\overline e}_t\,(\p_l^h\p_t^hF^h)(z).
\]
Therefore
\[
\sum_{l=0}^7 \bigl((\p_l^h\E^h)(z)\e_l\bigr)
=
\sum_{l,t=0}^7 (\p_l^h\p_t^hF^h)(z)\,\bm{\overline e}_t\e_l.
\]
Since mixed differences of the real-valued function \(F^h\) commute, the off-diagonal terms cancel pairwise because
\[
\bm{\overline e}_t\e_l+\bm{\overline e}_l\e_t=0,
\qquad l\neq t.
\]
Thus only the diagonal part survives, and we obtain
\[
\sum_{l=0}^7 \bigl((\p_l^h\E^h)(z)\e_l\bigr)
=
\sum_{l=0}^7(\p_l^h\p_l^hF^h)(z)
=
\Delta^h F^h(z)
=
\delta_0^h(z).
\]

Next we show that
\[
\sum\limits_{l=0}^7[(\p_l^h\E^h)(-x+y),\e_l,\f(x)]=0.
\]
Using again \(\E^h=\bm{\overline D}^{\,h}F^h\) and the reality of \(F^h\), we get
\[
\sum_{l=0}^7[(\p_l^h\E^h)(-x+y),\e_l,\f(x)]
=
\sum_{l=0}^7\sum_{t=0}^7(\p_l^h\p_t^hF^h)(-x+y)[\bm{\overline e}_t,\e_l,\f(x)].
\]
Since \(F^h\) is real-valued, the mixed differences commute:
\[
\p_l^h\p_t^hF^h=\p_t^h\p_l^hF^h.
\]
Moreover, for all \(l,t=0,\dots,7\) and all \(a\in\OO\), alternativity gives
\begin{equation}
    \label{2023-09-13 09:30:29}
    [\bm{\overline e}_l,\e_t,a]+[\bm{\overline e}_t,\e_l,a]=0.
\end{equation}
Hence the double sum cancels pairwise and therefore vanishes.

Combining these identities, we obtain
\[
G^h(x,y)=\delta_0^h(-x+y)\f(x)- \E^h(-x+y)(\bm D^h \f)(x).
\]
Integrating over \(B\) and using \eqref{2023-09-08 09:07:56}, we arrive at
\[
\chi_B(y)\f(y)=\int_{\p B}\bm K^h(x,y) * \f(x)dS(x)+ \int_B \E^h(y-x)(\bm D^h \f)(x)dV^h(x).
\]
This proves the theorem.
\end{proof}

        \begin{remark}  Theorem    \ref{2023-09-11 19:15:35}
        can be expanded to encompass general alternative algebras. However, this is not applicable to Cayley-Dickson algebras, due to the need for the alternative property of octonions as showcased in \eqref{2023-09-13 09:30:29}.        \end{remark}

     \begin{theorem}
           [Cauchy integral theorem]  Let $B$ be a bounded set in $\Z^8_h.$ Then for any function $\f\in M(\overline B,\OO)$ that is discrete regular in $B$, we have \[
\chi_B(y)\f(y)=\int_{\p B}\bm K^h(x,y) * \f(x)dS(x).
\]
        \end{theorem}

\begin{proof}
	This follows directly from the discrete Cauchy-Pompeiu integral formula. 	
\end{proof}

The discrete Cauchy-Pompeiu integral formula prompts further exploration, particularly the introduction of the discrete Cauchy-Bitsadze operator and the discrete Teodorescu operator.

        \begin{definition}
            Let $B$ be a bounded set in the lattice $\Z^8_h$.

            1. The {\bf discrete Cauchy-Bitsadze operator}   $\bm C^h_{\p B} : M(\p B, \OO) \longrightarrow M(\Z^8_h, \OO)$ is defined by \[
\bm C^h_{\p B}\f (y) := \int_{\p B} \bm K^h(x,y)*\f (x) dS(x).
\]

            2. The  {\bf discrete Teodorescu operator}  $\bm T^ h_ B : M(B, \OO) \longrightarrow M(\Z^8_h, \OO)$ is defined by \[
\bm T^h_B \f (y):= \int_{B} \E^h(y-x)\f(x) dV^h(x).
\]
        \end{definition}

\begin{theorem}
    [Inverse of Cauchy integral theorem]\label{thm:cauchy-integral-inverse}\label{2023-09-11 17:48:54} Let $B$ be a bounded set in $\Z^8_h,$ and $\f\in M( \overline B,\OO).$ The following assertions hold:

   \begin{itemize}
    \item  The Cauchy type integral   $ \bm C^h_{\p B}\f$
      is regular both in $B^{\circ}$ and in  $(\Z^8_h\setminus B)^{\circ}.$

     \medskip

    \item   If there holds
                $ \f(y)=\bm C^h_{\p B}\f(y)$  for any $y\in   B,$
      then  $\f$ is regular in $B^\circ$.
      \end{itemize}

\end{theorem}
\begin{proof}
    Since \(B\) is finite, differentiation may be passed through the boundary sum. It therefore suffices to prove that, for every \(x\in \p B\) and every \(y\in B^{\circ}\cup (\Z^8_h\setminus B)^{\circ}\),
    \[
    \bm D_y^h\bigl(\bm K^h(x,\cdot)*\f(x)\bigr)(y)=0.
    \]
    Fix \(x\in \p B\). Because \(y\in B^{\circ}\cup (\Z^8_h\setminus B)^{\circ}\), neither \(he_l-x+y\) nor \(-he_l-x+y\) can vanish. Hence \((\bm D_y^h\E^h)(he_l-x+y)=(\bm D_y^h\E^h)(-he_l-x+y)=0\) for every \(l\).

    For a fixed lattice point \(z\neq 0\) and a fixed octonion \(c\), the constant-right-factor rule gives
    \[
    \bm D_y^h\bigl(\E^h(z+y)\,c\bigr)
    =
    (\bm D_y^h\E^h)(z+y)\,c
    -
    \sum_{j=0}^7[\e_j,(\p_j^h\E^h)(z+y),c].
    \]
    Since \((\bm D_y^h\E^h)(z+y)=0\), it remains to examine the associator sum. Writing again \(\E^h=\bm{\overline D}^{\,h}F^h\) with \(F^h\) real-valued, we obtain
    \[
    (\p_j^h\E^h)(z+y)=\sum_{t=0}^7 \bm{\overline e}_t\,(\p_j^h\p_t^hF^h)(z+y),
    \]
    and therefore
    \[
    \sum_{j=0}^7[\e_j,(\p_j^h\E^h)(z+y),c]
    =
    \sum_{j,t=0}^7 (\p_j^h\p_t^hF^h)(z+y)\,[\e_j,\bm{\overline e}_t,c].
    \]
    Since the mixed differences commute and
    \[
    [\e_j,\bm{\overline e}_t,c]+[\e_t,\bm{\overline e}_j,c]=0,
    \qquad c\in\OO,
    \]
    the double sum cancels pairwise. Thus
    \[
    \bm D_y^h\bigl(\E^h(z+y)\,c\bigr)=0
    \qquad\text{for every }z\neq 0.
    \]

    Applying this with \(c=n_l^-(x)\e_l\f(x)\) and \(z=he_l-x\), and also with \(c=n_l^+(x)\e_l\f(x)\) and \(z=-he_l-x\), we conclude that each summand in the kernel is discrete regular as a function of \(y\). Therefore \(\bm C^h_{\p B}\f\) is discrete regular on \(B^{\circ}\) and on \((\Z^8_h\setminus B)^{\circ}\).

    For the converse statement, if \(\f(y)=\bm C^h_{\p B}\f(y)\) for every \(y\in B\), then \(\f\) coincides on \(B^{\circ}\) with a discrete regular function, and hence \(\f\) is discrete regular on \(B^{\circ}\).
\end{proof}

As in the continuous setting, the Teodorescu operator $\bm T^ h_ B$ is a right inverse of the Cauchy-Fueter operator $\bm D^h$:

\begin{theorem}\label{thm:teodorescu-right-inverse}\label{2023-09-29 15:22:41}
    Let $B$ be a bounded set in $\Z^8_h$. Then for any $\f\in M(B,\OO)$, there holds
    \[
\bm D^h(\bm T^h_B \f)(y)= \chi_B(y)\f(y).
\]
\end{theorem}

\begin{proof}
    Since $B$ is finite, differentiation may be passed through the sum. Hence
    \[
    \begin{aligned}
        \bm D^h(\bm T^h_B \f)(y)
        &= \sum_{x\in B}\sum_{l=0}^7 \e_l\bigl(\p_l^h \E^h\bigr)(y-x)\f(x)\,h^8\\
        &= \int_B (\bm D_y^h \E^h)(y-x)\f(x)\,dV^h(x)
           - \int_B \sum_{l=0}^7 [\e_l,(\p_l^h\E^h)(y-x),\f(x)]\,dV^h(x).
    \end{aligned}
    \]
    The first term equals \(\chi_B(y)\f(y)\) because \(\bm D^h\E^h=\delta_0^h\). For the second term we use
    \[
    \E^h=\bm{\overline D}^{\,h}F^h,
    \qquad
    F^h(z)=\frac{1}{h^6}F^1\!\left(\frac{z}{h}\right),
    \]
    where \(F^h\) is real-valued. Therefore
    \[
    (\p_l^h\E^h)(y-x)=\sum_{t=0}^7 \bm{\overline e}_t\,(\p_l^h\p_t^hF^h)(y-x),
    \]
    and thus
    \[
    \sum_{l=0}^7 [\e_l,(\p_l^h\E^h)(y-x),\f(x)]
    =
    \sum_{l,t=0}^7 (\p_l^h\p_t^hF^h)(y-x)\,[\e_l,\bm{\overline e}_t,\f(x)].
    \]
    Since \(F^h\) is real-valued, the mixed differences commute:
    \[
    \p_l^h\p_t^hF^h=\p_t^h\p_l^hF^h.
    \]
    Moreover, alternativity implies
    \[
    [\e_l,\bm{\overline e}_t,a]+[\e_t,\bm{\overline e}_l,a]=0,
    \qquad a\in\OO.
    \]
    Hence the double sum cancels pairwise, so the associator contribution vanishes identically. Therefore
    \[
    \bm D^h(\bm T^h_B \f)(y)=\chi_B(y)\f(y).
    \]
\end{proof}

\section{Estimation of the Discrete Cauchy Kernel}

Let us consider the discrete Cauchy kernel in the lattice structure. We aim to provide a quantitative estimate for this kernel, which proves to be pivotal in the context of the discrete regular extension theory. Specifically, the relevance of this estimate can be observed in Theorems \ref{2023-09-29 17:27:26} and \ref{2023-09-29 18:34:59}.

Let $\mathcal F_{\R^8}$, $\mathcal F_{\mathbb T^8}$ be the Fourier transforms on $\R^8$ and on $\mathbb T:=[-\pi,\pi]^8$ respectively, i.e. \begin{equation*}
    \begin{aligned}
        \mathcal F_{\R^8} g (v)&=\frac{1}{(2\pi)^8}\int_{\R^8} g(u)e^{iu\cdot v} du,\\
        \mathcal F_{\mathbb T^8} g (v)&=\frac{1}{(2\pi)^8}\int_{\mathbb T^8} g(u)e^{iu\cdot v} du.
    \end{aligned}
\end{equation*}

It is known that the fundamental solution of the Dirac operator in octonions
\[
D=\sum_{l=0}^7 e_l\partial_l
\]
is given by
 \[
\E(x) = \dfrac{3}{\pi^4} \dfrac{\bm {\overline x}}{|x|^8}.
\]
For its discrete version, we know that the discrete Dirac operator
\[
D^h=\sum_{l=0}^7 e_l\partial_l^h
\]
has the fundamental solution $E^h$, given in  \eqref{2023-09-07 10:25:17}. We call it the discrete Cauchy kernel of octonions.
Without loss of generality, we may assume $h=1$.

Within the Fourier domain, we consider   the discrete and continuous fundamental solutions above, which  take the form
\begin{equation}
	\label{2023-10-02 15:12:13}
	\begin{aligned}
	 \bm{S_c}	 & = \mathcal{F}_{\R^8}^{-1} \E, \\
	\bm{S_d}	& = \mathcal{F}_{\mathbb T^8}^{-1}  \E^1.
	\end{aligned}
\end{equation}
By direct calculation, they can be expressed explicitly as
\[ \begin{split}
	\bm{S_d}(u) & = \dfrac{\sum_{l=0}^7 \overline{\e_l} \, \sin u_l}{i\sum_{l=0}^7 \sin^2 u_l}, \\
	\bm{S_c}(u) & = \dfrac{\sum_{l=0}^7 \overline{\e_l} \, u_l}{i|u|^2}.
\end{split} \]

These two symbols have very different singular sets. The continuous symbol \( \bm{S_c} \) has a single singularity at \(u=0\), whereas the periodic symbol \( \bm{S_d} \) is singular at every corner point
\[ u = \pi \sum_{l=0}^7 \alpha_l e_l, \qquad \alpha_l \in \{0,\pm 1\}. \]

We will express the discrete Cauchy kernel in terms of the continuous one together with the following auxiliary functions.

1. Let \( \phi:\mathbb R^8\to \mathbb R \) be a smooth cutoff function such that
\[ \phi(u) = \begin{cases}
	1, & \text{if } u \in \left[-\frac{\pi}{8},\frac{\pi}{8}\right]^8, \\
	0, & \text{if } u \notin \left[-\frac{\pi}{4},\frac{\pi}{4}\right]^8.
\end{cases} \]

2. Define \( \Phi:\mathbb R^8\to \OO \) by
  \[
\Phi(u) = (\phi(u) - 1) \bm{S_c}(u).
\]

3. We consider an error function
\[ \Theta(u) = \bm{S_d}(u) - \sum_{\substack{\alpha_l \in \{0, \pm 1\}, \\ 0 \leq l \leq 7}} \phi \left( u + \pi \sum_{l=0}^7 \alpha_l e_l \right) \bm{S_c} \left( u + \pi \sum_{l=0}^7 \alpha_l e_l \right) \]

4. Introducing a bounded weight function:
\[ \omega(x) = \prod_{l=0}^7 \left( 1 + 2\cos(\pi x_l) \right) \]

Now we can provide close relation between the discrete and continuous Cauchy kernels.

\begin{theorem}\label{2023-09-30 10:43:11}
  The fundamental solution $\bm E^1$ admits the following expansion: \begin{equation}
        \label{2023-09-30 11:08:22}
        \bm E^1(x)= \omega(x) \bm E(x) + \frac{1}{|x|^8} \omega(x) \mathcal F_{\R^8} (\Delta^4 \Phi)(x) + \frac{1}{|x|^8} \mathcal F _{\mathbb T^8} (\Delta^4 \Theta) (x)
    \end{equation} for any $x\in   \Z^8\setminus \{0\},$   In particular, \[
\bm E^1(x) = \omega(x) \bm E(x) +O(|x|^{-8})
\]
 as  $|x|\rightarrow +\infty$.
\end{theorem}

The proof of  \Cref{2023-09-30 10:43:11} relies heavily on   the remarkable  work by Shivakumar and Wong
  \cite{SW1979}
  concerning the asymptotic
  expansion of Fourier transforms.
   We denote \[
L_{q}= 2^{q+n/2}  \Gamma\left(\dfrac{q+n}{2}\right) \left(\Gamma\left(\dfrac{-q}{2}\right)\right)^{-1}.
\]

\begin{theorem} [\cite{SW1979}] \label{2023-09-30 17:52:36}
    Let function f $\in C^\infty(\R^n \setminus\{0\})$ admit an expansion of a finite sum
    \[
f(u)=\sum\limits_{p,q} c_{p,q} u^p |u|^q + \Phi(u),
\] where  \[
q+n>0
\] for all the real numbers q under the summation, $p =(p_1,\cdots,p_n)$ is a multi-index of non-negative integers,
    $u^p=u_1^{p_1}\cdots u_n^{p_n}$,
     and $\Phi \in C^\infty(\R^n).$  Then there holds \[
\mathcal F_{\R^n}   f(v)= \sum\limits_{p,q} c_{p,q}L_{q}
    \left (\frac{1}{i} \frac{\partial}{\partial v_1} \right)^{p_1} \cdots \left(\frac{1}{i} \frac{\partial}{\partial v_n} \right)^{p_n}
      (|v|^{-q-n}) + \dfrac{(-1)^m}{|v|^{2m}} \mathcal F_{\R^n} (\Delta ^m \Phi)(v)
\] provided the following assumptions hold true:

    (1) We assume \[
2m-1\leqslant Q +n \leqslant 2m +1
\] where \[
Q = \mathop{\max}\limits_{p,q} (q +|p|)
\] with the maximum taken over all the multi-indices $p$ and the real number $q$ under the summation.

    (2) For any $j = 0, 1,...,m,$   assume that, as $u \rightarrow 0,$ \[
\Delta^j \Phi(u)= \begin{cases}
        O(|u|^{Q-2j+2}),& Q+n=2m-1,\\\\
        O(|u|^{Q-2j+1}),&Q+n\neq 2m-1.
    \end{cases}
\]

    (3) There exists a constant $\rho>0$ such that, for any $j = 0, 1,...,m,$ the limit
        \[
\lim_{R\rightarrow +\infty} \int_{\rho < |u| < R} \Delta^j f(u) e^{iu\cdot v} d u
\]
    converges uniformly with respect to $v$ for sufficiently large $|v|$.
\end{theorem}

\begin{proof}[Proof of \Cref{2023-09-30 10:43:11}] The expansion  \eqref{2023-09-30 11:08:22} can be derived by the following two claims:

    \textbf{claim 1}: \[
\bm E^1(x)= \omega(x) \mathcal F_{\R^8} (\phi \bm S_c) (x)+ \dfrac{1}{|x|^8} \mathcal F_{\mathbb T^8} (\Delta^4 \Theta ) (x),
\]

    \textbf{claim 2}: \[
\mathcal F_{\R^8} (\phi \bm S_c)(x) = \bm E(x) + \dfrac{1}{|x|^8} \mathcal F_{\R^8} (\Delta^4\Phi) (x).
\]

    We prove claim 1 first. By \eqref{2023-10-02 15:12:13}, $\E^1 = \mathcal F_{\mathbb T^8} \bm S_d,$ which means \[
\E^1= \mathcal F_{\mathbb T^8} \left(  \sum\limits_{\substack{\alpha_l=0,\pm 1,\\ 0\leqslant l\leqslant 7}} \phi(u+\pi\sum\limits_{l=0}^7 \alpha_l e_l) \bm S_c (u+\pi\sum\limits_{l=0}^7 \alpha_l e_l)   \right) + \mathcal F_{\mathbb T^8} \Theta.
\] A direct calculation shows that    the first summand on the right side equals $\omega \mathcal F_{\R^8} (\phi \bm S_c). $ On the other hand, since all the derivatives of $\Theta$ up to the eighth order belong to $L^1([-\pi,\pi]^8)$, we have \[
\mathcal F_{\mathbb T^8} \Theta = \frac{1}{|x|^8} \mathcal F_{\mathbb T^8 } (\Delta ^4 \Theta).
\] Combining the results above yields \textbf{claim 1}.

    Now we prove \textbf{claim 2}.  In \Cref{2023-09-30 17:52:36}, we take \[
\bm \Phi= (\phi-1)\bm S_c,
\]  and \[
\f(u)= (\phi \bm S_c)(u)= \dfrac{\sum_{l=0}^7\overline\e_l\, u_l}{i|u|^2} + \bm \Phi(u),
\] also take the parameters
    \[
n=8,  \quad m=4, \quad Q= -1, \quad |p|=1, \quad q=-2.
\] Then we obtain \[
\begin{aligned}
        \mathcal F_{\R^8}(\phi\bm S_c)(x)&= \frac{1}{2\pi^4i} \left( \dfrac{1}{i}\sum\limits_{l=0}^7 \overline \e_l\dfrac{\p}{\p x_l} \right) |x|^{-6} + \frac{1}{|x|^8} \mathcal F_{\R^8} (\Delta^4 \bm  \Phi) (x)\\&=\dfrac{3}{\pi^4}\dfrac{\bm {\overline  x}}{|x|^8} + \frac{1}{|x|^8} \mathcal F_{\R^8} (\Delta^4  \bm \Phi) (x)\\&=\E(x) +  \frac{1}{|x|^8} \mathcal F_{\R^8} (\Delta^4  \bm \Phi) (x).
    \end{aligned}
\]
    This finishes the proof.
\end{proof}

\Cref{2023-09-30 10:43:11} gives an estimate for  the discrete fundamental solution $\E^1$. For a more standard estimate for this function, we stipulate that $\bm E(0):=0$ to get the following result.

\begin{theorem}\label{2023-09-30 19:16:04}
    For any $h > 0$ and $x \in \Z^8_h$, we have \begin{equation}
        |\E^h(x)-\omega(h^{-1}x)\E(x)|\leqslant \dfrac{Ch}{|x|^8+h^8}.
    \end{equation}
\end{theorem}

\begin{proof}
    First, we take $h = 1.$ Then by \Cref{2023-09-30 10:43:11}, we have \[
|\E^1(x)-\omega(x)\E(x)|\leqslant \dfrac{C_0}{|x|^8}, \quad x\in \Z^8\setminus \{0\}
\] where \[
C_0=\frac{1}{\pi^8}\int_{\R^8} |\Delta^4 \Phi(u)| du + \frac{1}{(2\pi)^8} \int_{[-\pi,\pi]^8} |\Delta^4 \Theta(u)| du.
\]

    Furthermore, \[
|\E^1(0)-\omega(0)\E(0)|=|\E^1(0)|\neq +\infty.
\]
    Thus \begin{equation}
        \label{2023-09-30 20:07:46}
        |\E^1(x)-\omega(x)\E(x)|\leqslant \dfrac{C}{|x|^8+1}
    \end{equation}  with $C=\max\{ |\E^1(0),2C_0\}.$

    By the definitions of the discrete and continuous fundamental solutions, we have \begin{equation*}
            \E^h(x)=h^{-7}\E^1(h^{-1}x), \quad x\in \Z^8_h,
    \end{equation*}
    and \begin{equation*}
            \E(x)=h^{-7}\E(h^{-1}x), \quad x\in \R^8.
    \end{equation*}
    Therefore, \[
|\E^h(x)-\omega(h^{-1}x)\E(x)|=h^{-7}|\E^1(x^\prime)-\omega(x^\prime) \E(x^\prime)|
\] with $x^\prime=h^{-1}x. $ Hence \eqref{2023-09-30 20:07:46} yields that \[
|\E^h(x)-\omega(h^{-1}x)\E(x)|\leqslant \dfrac{Ch^{-7}}{|x^\prime|^8+1}= \dfrac{Ch }{|x|^8+h^8}.
\]
\end{proof}

\section{Sokhotski-Plemelj formula}

This section delves into the Sokhotski-Plemelj formula, spotlighting operators like the discrete Cauchy-Bitsadze and Teodorescu. These operators, tied to the lattice's boundaries, interlink functions on boundaries with those spanning the lattice. Key properties, like the discrete Cauchy-Pompeiu formula, enable function decomposition. Furthermore, the Sokhotski-Plemelj theorem and the discrete Plemelj projections underscore the significance of boundary values in discerning function behaviors, solidifying our understanding of functions on lattices.

	In parallel with the continuous context, the Teodorescu operator \(\bm T ^h_ B\) serves as a right inverse to the Cauchy-Fueter operator \(\bm D^h\). Formally, for any \(\f \in M(B, \mathbb O)\),
	\[
	\bm D^h \bm T^h_B \f = \chi_B\f.
	\]
	
	Leveraging these notations, the discrete Cauchy-Pompeiu formula (refer to \Cref{2023-09-11 19:15:35}) can be articulated as follows:
	\begin{equation}
		\label{2023-09-29 10:06:18}
		\bm C^h_{\p B} \f + \bm T^h_B (\bm D^h \f ) = \chi_B \f.
	\end{equation}

	\begin{definition}
		Given a bounded domain \(B\) in the lattice \(\Z^8_h\), we define an integral operator \(\bm S^h_{\p B}   : M(\p  B, \OO) \to M(\p B, \OO)\) by
		\[
		\bm S^h_{\p B} \f (y) = 2\int_{\p B} \bm K^h (x,y) \ast (\f(x)-\f(y)) \, dS(x) + \f(y).
		\]
	\end{definition}
	
	\begin{notation}
		We represent the inner and outer boundaries of \(B\) as \(\p^+ B\) and \(\p^- B\) respectively, where
		\[
		\p^+ B = \p B \cap B \quad \text{and} \quad \p^- B = \p B \setminus B.
		\]
	\end{notation}

	\begin{theorem}[Discrete Sokhotski-Plemelj formula]\label{2023-09-29 10:56:55}
		Given a bounded set \(B\) in \(\Z^8_h\) and any function \(\f \in M(\p B, \OO)\), the following relation holds:
		\[
		\bm C^h_{\p B} \f (y) = \frac{1}{2}\left( \pm \f(y) + \bm S^h_{\p B} \f(y) \right) \quad \forall y \in \p^{\pm} B.
		\]
	\end{theorem}
	
\begin{proof}
	Fix \(y\in \p B\), and let \(c_y\) denote the constant function on \(\overline B\) given by \(c_y(x)=\f(y)\). Since \(\bm D^h c_y=0\), the discrete Cauchy--Pompeiu formula yields
	\[
	\bm C^h_{\p B} c_y(y)=\chi_B(y)\f(y).
	\]
	Using the definition of \(\bm S^h_{\p B}\), we therefore obtain
	\[
	\begin{aligned}
		\bm S^h_{\p B}\f(y)
		&=2\bm C^h_{\p B}(\f-c_y)(y)+\f(y)\\
		&=2\bm C^h_{\p B}\f(y)-2\chi_B(y)\f(y)+\f(y).
	\end{aligned}
	\]
	Rearranging gives
	\[
	\bm C^h_{\p B}\f(y)
	=
	\frac12\Bigl(\bm S^h_{\p B}\f(y)+(2\chi_B(y)-1)\f(y)\Bigr).
	\]
	If \(y\in \p^+B\), then \(\chi_B(y)=1\); if \(y\in \p^-B\), then \(\chi_B(y)=0\). Consequently,
	\[
	\bm C^h_{\p B} \f (y) = \frac{1}{2}\left( \pm \f(y) + \bm S^h_{\p B} \f(y) \right)
	\quad \forall y \in \p^{\pm} B.
	\]
\end{proof}

	\begin{definition}
		We define the discrete Plemelj projections by
		\begin{align*}
			\bm P^h _{\p B} \f &= \frac{1}{2}(I+\bm S^h_{\p B})\f, \\
			\bm Q^h_{\p B} \f &= \frac{1}{2}(I- \bm S^h_{\p B})\f.
		\end{align*}
	\end{definition}
	
	Using \Cref{2023-09-29 10:56:55}, the following representation holds:
	\begin{equation}\label{2023-09-29 10:57:15}
		\bm C^h_{\p B} \f(y) =
		\begin{cases}
			\bm P^h_{\p B} \f(y), & \text{if } y \in \p^+ B, \\
			\bm P^h_{\p B} \f(y) - \f(y), & \text{if } y \in \p^- B.
		\end{cases}
	\end{equation}
	This relation reflects the discontinuity between the inner and outer boundaries.

\begin{theorem}[Discrete regular extensions]\label{2023-09-29 17:27:26}
For  a bounded subset \(B\) of \(\Z^8_h\) and any function \(\f : \p B \rightarrow \OO\), we have
	\begin{enumerate}
		\item The projection $\bm P^h_{\p B} \f$ extends into $B$ as a discrete regular function.
		\item The projection $\bm Q^h_{\p B} \f$ extends into $\Z^8_h \setminus B$ as a discrete regular function vanishing at infinity.
	\end{enumerate}
\end{theorem}

\begin{proof}
(i) 	Let the zero-extension of $\f$ over the space $\Z^8_h$ also be denoted by $\f$. From \eqref{2023-09-29 10:06:18}, we deduce
	\[
	\bm C^h_{\p B}\f(y) =
	\begin{cases}
		\f(y)- \bm T^h_B(\bm D^h \f )(y), & y\in \p^+B, \\
		- \bm T^h_B(\bm D^h \f )(y), & y\in \p^-B.
	\end{cases}
	\]
	Combining with \eqref{2023-09-29 10:57:15}, we find
	\[
	\bm P^h_{\p B}\f(y) = \f(y) - \bm T^h_B(\bm D^h \f )(y), \qquad \forall\ y\in \p B.
	\]
	Invoking \Cref{2023-09-29 15:22:41}, for any $\g\in M(B,\OO)$, we have:
	\[
	\bm D^h(\bm T^h_B \g)(y) = \chi_B(y) \g(y).
	\]
	This implies that the function \(\f(y) - \bm T^h_B(\bm D^h \f )(y)\) mapping \(\overline B \rightarrow \OO\) is discrete regular. Hence, \(\bm P^h_{\p B}\f(y)\) is extendable into \(\overline B\) as a discrete regular function.
	
(ii)  Since
\[
\bm P^h_{\p B} \f + \bm Q^h_{\p B} \f = \f,
\] we obtain
	\[
	\bm Q^h_{\p B} \f(y) = \bm T^h_B(\bm D^h \f )(y),   \quad \forall\   y\in \p B.
	\]
	The projection \(\bm Q^h_{\p B} \f(y)\) then corresponds to the boundary values of the function
	\[
\bm T^h_B(\bm D^h \f ):\overline{\Z^8_h\setminus B} \rightarrow \OO.
\]
	Using \Cref{2023-09-29 15:22:41} again, this function is discrete regular.
	
	To show that \(\bm T^h_B(\bm D^h \f )(y)\) vanishes at infinity, consider its definition:
	\[
	\bm T^h_B(\bm D^h \f )(y) = \int_B \E^h(y-x)(\bm D^h \f)(x) dV^h(x).
	\]
	From \Cref{2023-09-30 19:16:04}, we can bound
\begin{eqnarray*}
	|\E^h(x)| &\leqslant & \left| \prod_{l=0}^7 (1+2\mbox{cos}(h^{-1} \pi x_l))\right| |E(x)|+ \frac{Ch}{|x|^8+h^8}
 \end{eqnarray*}
so that \begin{eqnarray}\label{estamae-850}
	|\E^h(x)|
	&\leqslant &
	\frac{C}{|x|^7} + \frac{Ch}{|x|^8+h^8}.
\end{eqnarray}
Here $C$ is a constant independent of $x$ and $h$.

	Consequently,
	\[
	\begin{split}
		|\bm T^h_B(\bm D^h \f )(y)| &\leqslant V^h(B) \max_{B} |\bm D^h f| \max_{x\in B} |\E^h(y-x)| \\
		&\leqslant V^h(B) \max_{B} |\bm D^h f| \left( \frac{C}{(|y| -\rho)^7 }+ \frac{Ch}{(|y| -\rho )^8+h^8} \right),
	\end{split}
	\]
	where \[
\rho=\max_{x\in B} |x|.
\]  This verifies that \(\bm T^h_B(\bm D^h \f )(y)\) vanishes at infinity, completing the proof.
\end{proof}

We now consider the Cauchy formula for discrete regular functions that vanish at infinity, pivotal for characterizing their boundary values.

\begin{lemma}[Cauchy formula for discrete regular functions]\label{2023-09-29 17:40:15}
	Given a bounded subset \(B\) of \(\Z^8_h\) and a discrete regular function \(\f: \overline{\Z^8_h\setminus B}\rightarrow \OO\) that vanishes at infinity, the following relation holds:
\begin{eqnarray}\label{estimate-094}
	-\bm C^h_{\p B} \f(y) = \chi_{\Z^8_h\setminus B}(y) \f(y).
\end{eqnarray}
\end{lemma}

\begin{proof} 	Let us extend $\f$ to the entirety of $\Z^8_h$ with zero outside its original domain, and still denote it by $\f$. Define
	\[
	\Omega_n := \left\{ y \in \Z^8_h : |y_k| \leqslant nh, \text{ for } k = 0,\dots,7 \right\}.
	\]
	
	Invoking the discrete Cauchy-Pompeiu formula from \eqref{2023-09-29 10:06:18}, we can write
	\[
	\bm C^h_{\p \Omega_n} \f = \chi_{\Omega_n} \f - \bm T^h_{\Omega_n} (\bm D^h \f) \quad \text{and} \quad \bm C^h_{\p B} \f = \chi_{B} \f - \bm T^h_{B} (\bm D^h \f).
	\]
	Subtracting these expressions leads to
	\begin{equation}\label{eq:Cauchy-234}
		\bm C^h_{\p \Omega_n} \f - \bm C^h_{\p B} \f = \chi_{\Omega_n\setminus B} \f - \bm T^h_{\Omega_n\setminus B} (\bm D^h \f),
	\end{equation}
	valid for sufficiently large values of $n$.
	
	Since  $\bm D^h \f$ vanishes on $\Omega_n\setminus B$, it follows that
	\[
	\bm T^h_{\Omega_n\setminus B} (\bm D^h \f) = 0,
	\]
	which simplifies \eqref{eq:Cauchy-234} to
	\begin{equation}
		\label{2023-09-29 16:30:18}
		\bm C^h_{\p \Omega_n} \f - \bm C^h_{\p B} \f = \chi_{\Omega_n\setminus B} \f.
	\end{equation}

	By letting $n$ go  to infinity   in \eqref{2023-09-29 16:30:18} and recognizing that $\chi_{\Omega_n\setminus B}$ converges to $\chi_{\Z^8_h\setminus B}$, our task is now reduced to proving that
	\[
	\lim_{n \to \infty} \bm C^h_{\p \Omega_n} \f(y) = 0
	\]
	for every fixed $y$.

	To approach this, note that the boundary $\p \Omega_n$ is defined by
	\[
	\p \Omega_n = \{x \in \Z^8_h : |x_k| \in \{nh, (n+1)h\}, \text{ and } |x_j| \leqslant nh \text{ for } j \neq k\}.
	\]
	Hence, there exists a constant $C$,  independent of $n$ and $h$, such that the discrete boundary measure $S$ on $\p \Omega_n$ satisfies
	\[
	S(\p \Omega_n) \leqslant Ch^7(n+1)^7.
	\]
	
	By  definition, the Cauchy operator is given by
	\begin{align*}
		\bm C^h_{\p \Omega_n} \f(y) &= \int_{\p \Omega_n} \bm K^h(x,y) * \f(x) \, dS(x) \\
		&= \int_{\p \Omega_n} -\frac{1}{2} \sum_{l=0}^7 \left(\E^h(he_l-x+y) n_l^-(x) + \E^h(-he_l-x+y) n_l^+(x) \right) (\e_l \f(x)) dS(x).
	\end{align*}
We define the maximum absolute value of the integrand as $\bm I$.
	  We thus have \begin{equation}\label{2023-09-29 17:02:04}
		\begin{aligned}
			|\bm C^h_{\p \Omega_n}\f (y)| \leqslant S(\p \Omega_n) \max_{x\in \p \Omega_n} |\f(x)|  \bm I &\leqslant Ch^7 (n+1)^7 \max_{x\in \p \Omega_n} |\f(x)|  \bm I.
		\end{aligned}
	\end{equation}

Since $|n_l^{\pm}(x)|\le 2$ on $\p \Omega_n$, it follows that
	\[
	\bm I \leqslant \max_{x \in \p \Omega_n} \sum_{l=0}^7 |\E^h(he_l-x+y)| + |\E^h(-he_l-x+y)|.
	\]
	For a fixed $y$ and using \eqref{estamae-850}, we deduce that
	\[
	\bm I \leqslant \left( \frac{16C}{(nh-|y|-h)^7} + \frac{16Ch}{(nh-|y|-h)^8+h^8} \right),
	\]
	which implies
	\begin{equation}\label{2023-09-29 17:09:29}
	  \bm I = O(n^{-7}).
	\end{equation}
	
	Since $\f$ vanishes at infinity, we have
	\begin{equation}\label{2023-09-29 17:09:48}
		\max_{x \in \p \Omega_n} |\f(x)| = o(1).
	\end{equation}
	
	Inserting \eqref{2023-09-29 17:09:29} and \eqref{2023-09-29 17:09:48} into \eqref{2023-09-29 17:02:04}, it becomes clear that
	\[
	|\bm C^h_{\p \Omega_n}\f(y)| = o(1)
	\]
	as $n \to \infty$, completing the proof.	
\end{proof}

\begin{theorem}\label{2023-09-29 18:34:59}
	Let $B$ be a bounded subset of $\Z^8_h$ and let $\f: \partial B \to\mathbb O$ be any  function. The following holds:
	
	\begin{enumerate}
		\item $\f$ is the boundary value of a discrete regular function on $B$ if and only if
		\[
		\bm P^h_{\p B}\f(y) = \f(y), \qquad \forall\ y \in \p B.
		\]
		
		\item $\f$ is the boundary value of a discrete regular function on $\Z^8_h\setminus B$ that vanishes at infinity if and only if
		\[
		\bm Q^h_{\p B}\f(y) = \f(y), \qquad \forall\  y \in \p B.
		\]
	\end{enumerate}
\end{theorem}

\begin{proof}
	The sufficiency is directly derived from \Cref{2023-09-29 17:27:26}, so we focus on establishing the necessity.
	
(i)	Suppose $\f : \p B \to \OO$ is the boundary value of a discrete regular function $\bm u : B \to \OO$. By equations (\ref{2023-09-29 10:06:18}) and (\ref{2023-09-29 10:57:15}), we get
	\[
	\bm C^h_{\p B} \bm u(y) = \bm u(y) \chi_B(y)
	\]
	and
	\[
	\bm P^h_{\p B} \f(y) =
	\begin{cases}
		\bm C^h_{\p B} \f(y), & y \in \p^+ B, \\
		\bm C^h_{\p B} \f(y) + \f(y), & y \in \p^- B.
	\end{cases}
	\]
	Combining the above, we deduce
	\[
	\bm P^h_{\p B} \f(y) = \f(y), \qquad \forall\  y \in \p B.
	\]
	
(ii) 	If $\f : \p B \to \OO$ represents the boundary value of a discrete regular function
{ $\bm u : \overline{\Z^8_h \setminus B} \to \OO$} vanishing at infinity, then \Cref{2023-09-29 17:40:15} states
	\[
	-\bm C^h_{\p B} \f(y) =
	\begin{cases}
		\bm u(y), & y \in \Z^8_h\setminus B, \\
		0, & y \in B.
	\end{cases}
	\]
	Furthermore, from \Cref{2023-09-29 10:56:55}, we have
	\[
	\bm Q^h_{\p B}\f(y) =
	\begin{cases}
		-\bm C^h_{\p B} \f(y), & y \in \p^- B, \\
		\f(y) - \bm C^h_{\p B} \f(y), & y \in \p^+ B.
	\end{cases}
	\]
	This results in $\bm Q^h_{\p B}\f(y) = \f(y)$ for any $y \in \p B$, completing the proof.
\end{proof}

\begin{definition}\label{2023-09-11 10:00:54}
For a bounded domain $B$ within $\R^8$,	
	a collection of sets $B^h\subset \Z_h^8$ with $h>0$ is said to converge to $B$, expressed as
   \[
\lim_{h\to 0} B^h= B,
\] when the following conditions hold:
    \begin{equation}
        \begin{aligned}
            \lim_{h\to 0} \max_{\alpha\in \p B }\min_{\beta\in \p B^h} \Vert \alpha-\beta\Vert=0,\\
            \lim_{h\to  0} \max_{\alpha\in \p B^h }\min_{\beta\in \p B } \Vert \alpha-\beta\Vert=0,\\
            \lim_{h\to  0} \max_{\alpha\in \overline B}\min_{\beta\in B^h } \Vert \alpha-\beta\Vert=0,\\
            \lim_{h\to  0} \max_{\alpha\in B^h }\min_{\beta\in \overline B } \Vert \alpha-\beta\Vert=0.
        \end{aligned}
    \end{equation}

Furthermore, if $B^h$ converges to $B$ and
\[
V^h(B^h\setminus B) = O(h^2),
\]
then we say that the family $B^h$ has \emph{exterior excess of order $O(h^2)$}.
\end{definition}

We next record that the canonical inner lattice approximation converges to $B$ and has zero exterior excess.

\begin{theorem}
	For a bounded domain  $B$ in $\R^8$, the sets
  \[
B^h \equiv (B \cap \Z_h^8)^\circ
\] converge to $B$. Moreover,
  \[
V^h(B^h\setminus B)=0
\]
  for every $h>0$.
\end{theorem}

\begin{proof} Let us  confirm the four convergence identities from \Cref{2023-09-11 10:00:54}.

    \bigskip
    (i) First we verify the fourth identity. Since $B^h$ is contained in $B$, we have \begin{equation*}
        \min_{\beta\in \overline B } \Vert \alpha-\beta\Vert=0, \quad \forall \,\alpha\in B^h
    \end{equation*}
    so that  \[
\lim_{h\to 0} \max_{\alpha\in B^h }\min_{\beta\in \overline B } \Vert \alpha-\beta\Vert=0.
\]

    \bigskip

    (ii) To verify the third identity, we aim to show that given any $\varepsilon>0$, there exists a $\delta>0$ ensuring \[
\overline B \subset B^h + B(0, 2\varepsilon	)
\]   for any $h \leqslant \delta$.

    For any given 	$\varepsilon>0$ and $x \in \overline  B$ one can find $\delta > 0$ and $x^\prime \in  B$ such that \[
B(x^\prime,\delta ) \subset\subset B \cap B(x, \varepsilon	).
\]
    Consequently,
    \[
B(x^\prime,\delta) \cap \Z^8_h \subset  (B \cap \Z^8_h)^\circ  = B^h
\]
     for sufficiently small
  $h$, implying
     \[
B(x^\prime,\delta) \cap B^h = B(x^\prime,\delta) \cap \Z^8_h \neq \emptyset
\] for sufficiently small
     $h$.    From the fact \[
B(x,\varepsilon 	) \cap B^h \supset B(x^\prime,\delta) \cap B^h,
\] it follows that \begin{equation}
        \label{2023-09-11 16:15:35}
        B(x,\varepsilon 	) \cap B^h \neq \emptyset,\quad \forall h\ll 1.
    \end{equation}
    Due to the compactness of $B$, we can find a sequence $x_1,\cdots,x_n$ such that
    \[
\overline B\subset \bigcup\limits_{l=1}^n B(x_l,\varepsilon).
\] As we have proved for each $x_l$ there exists $\eta_l > 0$ such that \[
B(x_l,\varepsilon)\cap B^h  \neq \emptyset
\] for any $h<\eta_l.$ Hence for any $h<\min\limits_{1\leqslant l\leqslant n}\{\eta_l\},$ we have  \[
\overline B \subset B^h + B(0, 2\varepsilon	).
\]
    \bigskip

    (iii) Now we come to prove the second identity. It suffices to prove that \[
\min_{y\in \p B} \Vert x-y\Vert \leqslant 2h
\] for any  $x\in \p B^h$.

    {\bf Case 1:}  $x\in \p^+ B^h:= \p B^h\cap B^h.$

    \bigskip
    For any $y\in \p^+(B\cap \Z^8_h),$ according to the definition of $\p^+(B\cap \Z^8_h),$ we have $y \in B$ and there is a point $y^\prime \in \Z^8_h\setminus B$ such that $\Vert y^\prime-y \Vert = h.$ The fact that $y \in B$ and $y^\prime \notin  B$ leads to the conclusion that \[
\mathop{\min}\limits_{\beta \in \p B}\Vert y-\beta \Vert \leqslant h
\]
   for all  $y\in \p^+(B\cap \Z^8_h)$.

    For any $x\in \p^+ B^h,$ there exists $y\in N(x)$ such that $y\notin B^h$. Since $B^h$ is the discrete interior of $B\cap \Z^8_h,$ we have $y\in \p^+(B\cap \Z^8_h),$ which implies \[
\min_{\beta \in \p B}\Vert x-\beta \Vert \leqslant \Vert x-y\Vert +\min_{\beta\in \p B}\Vert y-\beta \Vert \leqslant 2h.
\]

    \bigskip
   {\bf  Case 2: } $x\in \p^-B^h:=\p B^h\setminus B^h.$

    \bigskip
    From the fact that $B^h=(B\cap \Z^8_h)^\circ,$ it follows that \[
x\in \p^+(B\cap \Z^8_h).
\] As shown above in case 1, we find \[
\min_{\beta\in \p B} \Vert x-\beta \Vert \leqslant h.
\]

    \bigskip

    (iv) At the final step, we come to verify the first identity \[
\lim_{h\to 0} \max_{\alpha\in \p B }\min_{\beta\in \p B^h} \Vert \alpha-\beta\Vert=0.
\] Since $\p B$ is compact, it is sufficient to prove that for any 	 $\varepsilon>0$  and $x \in\p B$, there exists $\delta>0$ such that \[
\min_{\beta \in \p B^h} \Vert x-\beta \Vert < \varepsilon
\] for any $h<\delta.$

    Indeed, for any given 	$\varepsilon>0$ and $x \in \p B$, as shown in \eqref{2023-09-11 16:15:35}, there exists $\eta>0 $ such that \[
B(x,\varepsilon)\cap B^h\neq \emptyset
\] for any $h<\eta.$

    On the other hand, because $B(x,\varepsilon 	)\setminus B$ is open, there exists $\eta^\prime > 0$ such that \[
B(x,\varepsilon)\cap \Z^8_h\setminus \overline B \neq \emptyset
\]
   for all  $h<\eta^\prime$.

    When $h<\min\{\eta,\eta^\prime\}$, choose
    \[
        \alpha \in B^h\cap B(x,\varepsilon), \qquad \beta \in (\Z_h^8\setminus B^h)\cap B(x,\varepsilon).
    \]
    Because the lattice set \(B(x,\varepsilon)\cap \Z_h^8\) is connected by nearest-neighbor steps, there exists a sequence
    \[
        \{x_l\}_{l=1}^n \subset \Z_h^8 \cap B(x,\varepsilon)
    \]
    such that
    \[
        x_1=\alpha,\qquad x_n=\beta,\qquad \Vert x_l-x_{l+1}\Vert=h
    \]
    for all \(1\le l\le n-1\). Since \(x_1\in B^h\) and \(x_n\notin B^h\), there exists an index \(l_0\) with
    \[
        x_{l_0}\in B^h,\qquad x_{l_0+1}\notin B^h.
    \]
    By the definition of the discrete boundary, both \(x_{l_0}\) and \(x_{l_0+1}\) belong to \(\p B^h\). In particular,
    \[
        \min_{\gamma\in \p B^h}\Vert x-\gamma\Vert<\varepsilon.
    \]
    This proves the first convergence identity. Since \(B^h\subset B\), we also have \(V^h(B^h\setminus B)=0\), and the proof is complete.

\end{proof}

\subsection{Convergence of Scaling Limits} In this subsection,  we shall  provide our main result about the convergence of scaling limits.  To this end we need   two auxiliary lemmas below.
\begin{lemma}\label{2023-09-11 19:48:45}
  Let $B$ be a bounded open set in $\R^8$ and assume that $\f \in C^3(\R^8,\OO)$ is regular on $B$. Then we have \[
\max_{B} \Vert \bm D^h \f\Vert  = O(h^2).
\]
\end{lemma}
\begin{proof}
    Since \(\bm D\f=0\) on \(B\), we have
    \[
        \bm D^h\f(x)-\bm D\f(x)
        =
        \sum_{l=0}^7 \e_l\left(
        \frac{\f(x+he_l)-\f(x-he_l)}{2h}-\p_{x_l}\f(x)\right).
    \]
    Let \(g\) be any real-valued \(C^3\)-component of \(\f\). Taylor's formula in the \(x_l\)-variable gives
    \[
        \frac{g(x+he_l)-g(x-he_l)}{2h}-\p_{x_l}g(x)
        =
        \frac{h^2}{6}\,\p_{x_l}^3 g(x+\theta_{l,x}he_l)
    \]
    for some \(\theta_{l,x}\in(-1,1)\). Since \(B\) is bounded and \(\f\in C^3(\R^8,\OO)\), all third-order derivatives of the components of \(\f\) are bounded on a neighborhood of \(B\). Therefore
    \[
        \max_{x\in B}\left\Vert
        \frac{\f(x+he_l)-\f(x-he_l)}{2h}-\p_{x_l}\f(x)
        \right\Vert
        =O(h^2)
    \]
    for each \(l=0,\dots,7\). Summing over \(l\) yields
    \[
        \max_{x\in B}\Vert \bm D^h\f(x)-\bm D\f(x)\Vert=O(h^2).
    \]
    Since \(\bm D\f=0\) on \(B\), the claim follows.
\end{proof}

\begin{lemma}\label{2023-09-11 19:50:23}
    Let $B$ be a bounded open set in $\R^8$. If $B^h \subset B$ for any $h>0$, then  \[
\int_{B^h}  1dV ^h = O(1).
\]
\end{lemma}
\begin{proof}
    Since $B$ is bounded we may assume $B\subset\subset B(0,R).$ Therefore, we have $B^h \subset B(0,R)$ so that
    \[
\int_{B^h} 1 dV^h \leqslant \int_{\Z^8_h} \chi_{B(0,R)} dV^h.
\]
    The last integral is identical to a Riemann sum of the integral $\int_{B(0,R)} 1 dV$ where $dV$ is the Lebesgue measure on $\R^8$, so that
    \[
\lim_{h\to 0+} \int_{\Z^8_h} \chi_{B(0,R)} dV^h = \int_{B(0,R)}1\,dV < +\infty.
\]
    Thus we come to the conclusion that
    \[
\int_{B^h}1 dV^h = O(1).
\]
\end{proof}

Now we come to the first main theorem in this section.

\begin{theorem}\label{thm:scaling-limit-approximation}\label{2023-09-11 20:34:07}
    Suppose that $B$ is a bounded domain with $C^3$-boundary in $\R^8$, and let $B^h\subset \Z_h^8$ be a family of lattice sets converging to $B$ with exterior excess $O(h^2)$ in the sense of \Cref{2023-09-11 10:00:54}. Then for any regular function $\f \in C^3(\overline B,\OO),$ there exists a sequence of lattice functions $\f^h : \Z_h^8 \longrightarrow \OO$, discrete regular on \(B^h\), that converges to $\f$ in the sense that \[
\lim_{h\to 0+} \max_{ B^h \cap B} \Vert \f-\f^h\Vert =0.
\]
\end{theorem}
\begin{proof}
    Let $\bm {\tilde f}\in C^3_0(\R^8,\OO)$ be a compactly supported extension of $\f$, and define
    \begin{equation}
        \label{2023-09-11 17:33:45}
        \f^h(x):=\int_{\p B^h} \bm K^h(y,x)* \bm {\tilde f} (y)dS(y),
        \qquad x\in \Z_h^8.
    \end{equation}
    By \Cref{2023-09-11 17:48:54}, the function $\f^h$ is discrete regular in $(B^h)^\circ$.

    Choose $R>0$ so large that $\overline B\subset B(0,R)$ and $B^h\subset B(0,R)\cap \Z_h^8$ for all sufficiently small $h$.
    Fix $p\in (1,\frac{8}{7})$ and let $q=\frac{p}{p-1}$.
    By \Cref{2023-09-30 19:16:04} and the boundedness of $\omega$, there exists a constant $C>0$ such that
    \begin{equation}
        \label{2026-04-01 00:00:01}
        |\E^h(z)|\le \frac{C}{(|z|+h)^7}, \qquad z\in \Z_h^8.
    \end{equation}
    Consequently,
    \begin{equation}
        \label{2026-04-01 00:00:02}
        \sup_{0<h\le 1}\sup_{x\in B(0,R)\cap \Z_h^8}
        \int_{B(0,R)\cap \Z_h^8} |\E^h(x-y)|^p\,dV^h(y)<+\infty.
    \end{equation}
    Indeed, for $x\in B(0,R)\cap \Z_h^8$,
    \begin{equation*}
        \begin{aligned}
            \int_{B(0,R)\cap \Z_h^8} |\E^h(x-y)|^p\,dV^h(y)
            &\le C\sum_{z\in B(0,2R)\cap \Z_h^8}\frac{h^8}{(|z|+h)^{7p}}\\
            &=Ch^{8-7p}\sum_{w\in B(0,2R/h)\cap \Z^8}\frac{1}{(|w|+1)^{7p}},
        \end{aligned}
    \end{equation*}
    and the last expression is bounded uniformly in $h$ because $7p<8$.

    Applying the discrete Cauchy-Pompeiu formula \Cref{2023-09-11 19:15:35} on the lattice set $B^h$ to the restriction of $\bm {\tilde f}$ to $\overline{B^h}$, we obtain
    \begin{equation*}
        \f(x)-\f^h(x)=\int_{B^h} \E^h(x-y)\left(\bm D^h\bm{\tilde f}(y)\right) dV^h(y),
        \qquad x\in B^h\cap B.
    \end{equation*}
    Write the right-hand side as
    \begin{equation*}
        \f(x)-\f^h(x)=I_1(x)+I_2(x),
    \end{equation*}
    where
    \begin{equation*}
        \begin{aligned}
            I_1(x)&:=\int_{B^h\cap B} \E^h(x-y)\left(\bm D^h\bm{\tilde f}(y)\right) dV^h(y),\\
            I_2(x)&:=\int_{B^h\setminus B} \E^h(x-y)\left(\bm D^h\bm{\tilde f}(y)\right) dV^h(y).
        \end{aligned}
    \end{equation*}

    Since $\bm {\tilde f}$ is regular on $B$, \Cref{2023-09-11 19:48:45} yields
    \begin{equation}
        \label{2026-04-01 00:00:03}
        \max_{B^h\cap B}\Vert \bm D^h \bm{\tilde f} \Vert =O(h^2).
    \end{equation}
    By Hölder's inequality, \eqref{2026-04-01 00:00:02}, and the fact that
    $B^h\cap B\subset B(0,R)\cap \Z_h^8$, we obtain
    \begin{equation*}
        \begin{aligned}
            |I_1(x)|
            &\le \max_{B^h\cap B}\Vert \bm D^h \bm{\tilde f}\Vert
            \left(\int_{B(0,R)\cap \Z_h^8} |\E^h(x-y)|^p dV^h(y)\right)^{1/p}
            \left(\int_{B^h\cap B} 1\,dV^h\right)^{1/q}\\
            &=O(h^2),
        \end{aligned}
    \end{equation*}
    uniformly for $x\in B^h\cap B$.
    Here $\int_{B^h\cap B}1\,dV^h=O(1)$ because $B^h\cap B\subset B(0,R)\cap \Z_h^8$.

    On the other hand, $\bm {\tilde f}\in C^3_0(\R^8,\OO)$ implies
    \begin{equation}
        \label{2026-04-01 00:00:04}
        \max_{\R^8}\Vert \bm D^h\bm{\tilde f}\Vert =O(1).
    \end{equation}
    Since the family $B^h$ has exterior excess of order $O(h^2)$, we have
    \begin{equation}
        \label{2026-04-01 00:00:05}
        V^h(B^h\setminus B)=O(h^2).
    \end{equation}
    Applying Hölder's inequality again and using \eqref{2026-04-01 00:00:02}, \eqref{2026-04-01 00:00:04}, and \eqref{2026-04-01 00:00:05}, we infer
    \begin{equation*}
        \begin{aligned}
            |I_2(x)|
            &\le \max_{\R^8}\Vert \bm D^h \bm{\tilde f}\Vert
            \left(\int_{B(0,R)\cap \Z_h^8} |\E^h(x-y)|^p dV^h(y)\right)^{1/p}
            \left(V^h(B^h\setminus B)\right)^{1/q}\\
            &=O\!\left(h^{2/q}\right),
        \end{aligned}
    \end{equation*}
    uniformly for $x\in B^h\cap B$.
    Since $q<+\infty$, the latter term tends to zero as $h\to 0+$.

    Combining the bounds for $I_1$ and $I_2$, we conclude that
    \begin{equation*}
        \max_{x\in B^h\cap B}\Vert \f(x)-\f^h(x)\Vert \to 0
        \qquad (h\to 0+).
    \end{equation*}
    This completes the proof.
\end{proof}

Next, we demonstrate that when the scaling limit of discrete regular functions is present, it is indeed regular.

\begin{theorem}\label{thm:scaling-limit-regularity}\label{2023-09-11 20:33:47}
    Suppose $B$ is a bounded domain in $\R^8$ and $B^h \subset  \Z^8_h$ converges to $B$. If a function $\f \in C^1(B,\OO)$ is the scaling limit of discrete regular functions $\f^h : B^h \longrightarrow \OO,$ i.e., \[
\lim_{ h\to 0+} \max_{ B^h \cap B} \Vert \f-\f^h\Vert  =0,
\] then $\f$ is regular on $B$.
\end{theorem}

The proof of \Cref{2023-09-11 20:33:47} relies on the next two lemmas.

\begin{lemma}\label{2023-09-11 21:51:03}
    Let $B$ be a bounded open set in $\R^8$ and assume $B^h\subset \Z^8_h$ converges to $B$. Then for any $U \subset\subset B,$ there exists $\delta>0$ such that
    \[
U \cap \Z^8_h \subset B^h
\] whenever $h<\delta.$
\end{lemma}

\begin{proof}
    Without loss of generality, we assume that $U = B(x_0, R) \subset\subset B.$

    From the assumption that $B^h$ converges to $B$, we have \[
\lim_{h\to 0+} \max_{\alpha\in \overline{B}}\min_{\beta\in B^h}\Vert \alpha-\beta\Vert =0.
\] Hence, for any given $\varepsilon>0$, we have $x_0\in B^h+B(0,\varepsilon)$ when $h$ is sufficiently small. In particular, by taking $\varepsilon=R$, we obtain $U\cap B^h\neq \emptyset$.

    On the other hand, we denote \[
d=\inf_{\alpha\in U,\beta \in \p B}\Vert \alpha-\beta \Vert.
\]
    From the assumption that $B^h$ converges to $B$, it follows that \[
\lim_{h\to 0+}\max_{\alpha\in \p B^h}\min_{\beta\in \p B}\Vert \alpha-\beta \Vert=0.
\]
    For any $h$ sufficiently small, we thus have \[
\p B^h \subset \p B+B(0,d)
\] so that \[
\p B^h\cap U\subset (\p B+B(0,d))\cap U.
\] Since $(\p B+B(0,d))\cap U=\emptyset,$ we obtain \[
\p B^h\cap U=\emptyset.
\]

   We have now established that for sufficiently small $h$
    both conditions
     \[
U \cap B^h \neq \emptyset, \qquad  U \cap \p B^h = \emptyset
\]
  hold true.

    Now we have proved that, when $h$ is sufficiently small, \[
U \cap B^h \neq \emptyset
\] and \[
U \cap \p B^h = \emptyset
\] are both valid.

    Based on these facts, we come to show that \[
U\cap \Z^8_h \subset B^h
\]
    for all sufficiently  small $h$.

    Assume this is not valid, then there exists $h^*$ such that \begin{equation*}
        U \cap B^{h^*} \neq \emptyset, \quad U \cap \p B^{h^*} = \emptyset, \quad U\cap \Z_{h^*}^8\nsubseteq B^{h^*}.
    \end{equation*}
    From the fact that \[
\alpha,\beta\in U\cap \Z^8_{h^*}=B(x_0,R)\cap \Z^8_{h^*}
\] and $B(x_0,R)\cap \Z^8_{h^*}$ is discretely connected in the grid $\Z^8_{h^*}$, we can derive that there exists a sequence \[
\{x_k\}_{k=1}^m \subset U \cap \Z^8_{h^*}
\] such that
    \begin{equation*}
        x_1=\alpha,\quad x_m=\beta,\quad \Vert x_{k+1}-x_k\Vert =h^*
    \end{equation*}
    for each $k=1,2,\cdots,m-1.$

     Notice that \[
\alpha \in B^{h^*}, \qquad \beta\notin B^{h^*},
\] one can find $k^*$ with $1 \leqslant k^* \leqslant m - 1$ such that \begin{equation*}
        x_{k^*}\in B^{h^*},\quad  x_{k^*+1}\notin B^{h^*}.
    \end{equation*}
    According to the definition of discrete boundaries, $x_{k^*}$ and $x_{k^*+1}$ are both in $\p B^{h^*}$ . It violates the assumption that $\p B^{h^*}\cap U=\emptyset$. The proof is complete.
\end{proof}

\begin{lemma}\label{2023-09-11 22:31:45}
    Let $B$ be a bounded open set in $\R^8$ and $B^h \subset \Z^8_h$ converges to $B$. Then we have \[
\lim_{h\to 0+} \int_{B^h\cap B}\f dV^h=\int_B\f dV
\]
    for any $\f\in C_0(B,\OO).$
\end{lemma}
\begin{proof}
	Let's consider any function
	$\f$
  from the set
	 $C_0(B,\OO)$.
	 We can pinpoint a set
	$U$,  which is the finite union of balls, fulfilling the condition:
   \begin{equation}
        \label{2023-09-11 21:52:53}\mbox{supp} \f\subset U\subset\subset B.
    \end{equation}

From the evidence provided in \Cref{2023-09-11 21:51:03}, when
$h$ is sufficiently small, it is  clear that  \[
U\cap  \Z^8_h \subset B^h\subset \Z_h^8,
\]
   which implies the equality  $B^h\cap U=U\cap \Z^8_h.$

  Utilizing both the above relation and \eqref{2023-09-11 21:52:53}, we can deduce
  \begin{equation*}
        \int_{B^h\cap B}\f dV^h=\int_{B^h\cap U}\f dV^h=\int_{U\cap \Z^8_h}\f dV^h.
    \end{equation*}

It is worth noting that the latter integral serves as a Riemann sum for the Riemann integral
  $\int_U \f dV$. Hence we obtain  \begin{equation*}
        \lim_{h\to 0+}\int_{B^h\cap B}\f dV^h=\lim_{h\to 0+}\int_{U\cap \Z^8_h} \f dV^h=\int_{U}\f dV=\int_B \f dV.
    \end{equation*}
   This concludes our proof.
\end{proof}

We now turn our attention to proving \Cref{2023-09-11 20:33:47}.

\begin{proof}[Proof of \Cref{2023-09-11 20:33:47}]
    Assume that a sequence of discrete regular functions $\f^h : B^h \longrightarrow \OO$ converges uniformly on $B^h\cap B$ to a function $\f\in C^1(B,\OO)$. We prove that $\f$ is regular on $B$.

    Since $\f\in C^1(B,\OO)$, it is enough to verify that $\bm D\f=0$ in the sense of distributions. Equivalently, for every real-valued test function $\phi\in C_0^\infty(B,\R)$, we must show that
    \begin{equation}
        \label{2026-04-01 00:00:06}
        \int_B \sum_{l=0}^7 (\p_l\phi)(x)\,\e_l\f(x)\,dV(x)=0.
    \end{equation}
    Indeed, because $\phi$ is real-valued, integration by parts gives
    \[
    \int_B \sum_{l=0}^7 (\p_l\phi)(x)\,\e_l\f(x)\,dV(x)
    =-\int_B \phi(x)\,(\bm D\f)(x)\,dV(x).
    \]

    Fix such a test function $\phi$, and define
    \[
    g(x):=\sum_{l=0}^7 (\p_l\phi)(x)\,\e_l\f(x)\in C_0(B,\OO).
    \]
    By \Cref{2023-09-11 22:31:45},
    \[
    \int_{B^h\cap B} g(x)\,dV^h(x) \longrightarrow \int_B g(x)\,dV(x)
    \qquad (h\to 0+).
    \]
    Hence it is enough to prove that
    \[
    J_h:=\int_{B^h\cap B} g(x)\,dV^h(x) \longrightarrow 0.
    \]

    We decompose $J_h$ as
    \[
    J_h=A_h+B_h+C_h,
    \]
    where
    \[
    \begin{aligned}
        A_h&:=\int_{B^h\cap B}\sum_{l=0}^7 (\p_l^h\phi)(x)\,\e_l\f^h(x)\,dV^h(x),\\
        B_h&:=\int_{B^h\cap B}\sum_{l=0}^7 (\p_l^h\phi)(x)\,\e_l\bigl(\f(x)-\f^h(x)\bigr)\,dV^h(x),\\
        C_h&:=\int_{B^h\cap B}\sum_{l=0}^7 \bigl((\p_l\phi)(x)-(\p_l^h\phi)(x)\bigr)\,\e_l\f(x)\,dV^h(x).
    \end{aligned}
    \]

    We first show that $A_h=0$ for all sufficiently small $h$. Choose an open set $U$ such that
    \[
    \operatorname{supp}\phi \subset U \subset\subset B.
    \]
    Since central differences only shift by one lattice step, we also have $\operatorname{supp}(\p_l^h\phi)\subset U$ for all sufficiently small $h$. By \Cref{2023-09-11 21:51:03}, after shrinking $h$ if necessary, we may assume that $U\cap \Z_h^8\subset B^h$.

    For each fixed $l$, the function $\phi$ has compact support, so discrete summation by parts on $\Z_h^8$ gives
    \[
    \begin{aligned}
        \int_{\Z_h^8} (\p_l^h\phi)(x)\,\e_l\f^h(x)\,dV^h(x)
        &= \frac{h^8}{2h}\sum_{x\in \Z_h^8} \bigl(\phi(x+he_l)-\phi(x-he_l)\bigr)\e_l\f^h(x)\\
        &= -\frac{h^8}{2h}\sum_{x\in \Z_h^8} \phi(x)\,\e_l\bigl(\f^h(x+he_l)-\f^h(x-he_l)\bigr)\\
        &= -\int_{\Z_h^8} \phi(x)\,\e_l\,(\p_l^h\f^h)(x)\,dV^h(x).
    \end{aligned}
    \]
    Summing over $l=0,\dots,7$, and using the facts that $\operatorname{supp}\phi\subset U\cap \Z_h^8\subset B^h$ and $\bm D^h\f^h=0$ on $B^h$, we obtain
    \[
    A_h=-\int_{\Z_h^8}\phi(x)\,(\bm D^h\f^h)(x)\,dV^h(x)=0.
    \]

    Next, because $\phi\in C_0^\infty(B,\R)$, the discrete derivatives $\p_l^h\phi$ are uniformly bounded on $\R^8$ for $0<h\le 1$. Since $B^h\cap B\subset B(0,R)\cap \Z_h^8$ for some fixed $R>0$, \Cref{2023-09-11 19:50:23} yields
    \[
    \int_{B^h\cap B}1\,dV^h=O(1).
    \]
    Therefore,
    \[
    |B_h|\le C\,\max_{B^h\cap B}\Vert \f-\f^h\Vert\int_{B^h\cap B}1\,dV^h=o(1).
    \]

    Finally, since $\phi\in C_0^\infty(B,\R)$, the central differences converge uniformly to the corresponding derivatives:
    \[
    \max_{\R^8}|\p_l^h\phi-\p_l\phi|\to 0 \qquad (h\to 0+).
    \]
    Because $\f$ is bounded on a neighborhood of $\operatorname{supp}\phi$ and again $\int_{B^h\cap B}1\,dV^h=O(1)$, we obtain
    \[
    |C_h|\le C\,\max_{0\le l\le 7}\max_{\R^8}|\p_l^h\phi-\p_l\phi|\int_{B^h\cap B}1\,dV^h=o(1).
    \]

    We have shown that $A_h=0$, $B_h=o(1)$, and $C_h=o(1)$. Hence $J_h\to 0$, and therefore \eqref{2026-04-01 00:00:06} holds. This proves that $\bm D\f=0$ in the distributional sense. Since $\f\in C^1(B,\OO)$, it follows that $\f$ is regular on $B$.
\end{proof}

In conclusion, a function is deemed regular precisely when it stands as the scaling limit of discrete regular functions.

\begin{theorem}
	Suppose $B$ is a bounded domain with $C^3$-boundary in $\R^8$, and let $B^h\subset \Z_h^8$ be a family of lattice sets converging to $B$ with exterior excess $O(h^2)$ in the sense of \Cref{2023-09-11 10:00:54}. Then a function $\f \in C^3(\overline B,\OO)$ is regular if and only if there exist lattice functions $\f^h : \Z_h^8\longrightarrow \OO$ that are discrete regular on \(B^h\) and such that $\f$ is the scaling limit of $\f^h$.
\end{theorem}

\begin{proof}
The result is a direct consequence of merging Theorems \ref{2023-09-11 20:34:07} and \ref{2023-09-11 20:33:47}.
\end{proof}

\bigskip\bigskip
\bigskip\bigskip

\noindent\textbf{Declarations.}

\bigskip\bigskip

\noindent {\bf Funding, conflicts of interest, and competing interests:} 
This work was supported by the NNSF of China (12571090).
 The authors declare that they have no conflicts of interest and no competing interests.  The authors read and approved the final manuscript.


\bigskip\bigskip


\begin{thebibliography}{99}

\bibitem{ACS} D.  Alpay, F. Columbo,K. Diki, I.  Sabadini, Volok D.   Discrete analytic functions, structured matrices and a new family of moment problems. Bull. Sci. Math. 179 (2022), Paper No. 103175, 70 pp.





\bibitem{B2002}
J. Baez, \emph{The octonions},
Bull. Amer. Math. Soc.,   \textbf{39}(2): 145-205 (2002).







\bibitem{BL}
 I. Benjamini,  L. Lov\'asz,   Harmonic and analytic functions on graphs. Combinatorics,   J. Geom. 76 (2003), no. 1-2, 3–15.












\bibitem{BMS}
A. I. Bobenko, C. Mercat, and Y. B. Suris, Linear and nonlinear theories of discrete analytic
functions. Integrable structure and isomonodromic Green’s function, J. Reine Angew. Math.
583 (2005), 117–161.




\bibitem{GP1}    G. M.  Cabrera,   P. \'A Rodríguez,  Non-associative normed algebras. Vol. 1. The Vidav-Palmer and Gelfand-Naimark theorems. Encyclopedia of Mathematics and its Applications, 154. Cambridge University Press, Cambridge, 2014.






\bibitem{GP2}     G. M. Cabrera,   P. \'A Rodríguez,  Non-associative normed algebras. Vol. 2. Representation theory and the Zel'manov approach. Encyclopedia of Mathematics and its Applications, 167. Cambridge University Press, Cambridge, 2018.



\bibitem{CKK}
 P. Cerejeiras, U. K\"ahler, M. Ku,   Discrete Hilbert boundary value problems on half lattices. J. Differ. Equ. Appl. 21(12), 1–28 (2015)




\bibitem{CKS}
P. Cerejeiras,  U.  K\"ahler, M.  Ku,  F.  Sommen,   Discrete Hardy spaces. J. Fourier Anal. Appl. 20,
715–750 (2014)


    \bibitem{CS2011}
    D. Chelkak, S. Smirnov,  { Discrete complex analysis on isoradial   graphs},
     Adv. Math. \textbf{228}, 1590-1630 (2011).


    \bibitem{CS2012}
    D. Chelkak, S. Smirnov,
    {Universality in the 2D Ising model and conformal invariance of fermionic observables}, Invent. Math.  {189}(3), 515-580 (2012).


\bibitem{CFL1928}
R. Courant, K. Friedrichs, H. Lewy,
{ Uber die partiellen Differenzengleichungen der mathematischen Physik}, Math. Ann.  {100}, 32-74 (1928).

\bibitem{CSSS2004}
F. Colombo,   I. Sabadini,   F. Sommen, D. C. Struppa,  Analysis of Dirac systems and computational algebra. Progress in Mathematical Physics, 39. Birkh\"{a}user Boston, Inc., Boston, MA, 2004.









\bibitem{Du}
R. J. Duffin,  Basic properties of discrete analytic functions. Duke Math. J. 23(2), 335–363 (1956)






\bibitem{Ferrand}
J. Ferrand,  Fonctions pr\'{e}harmoniques et fonctions pr\'{e}holomorphes. Bull. Sci. Math. (2) 68 (1944), 152-180.






\bibitem{G}
A. Grigor'yan,   Introduction to analysis on graphs. University Lecture Series, 71. American Mathematical Society, Providence, RI, 2018.




\bibitem{Gu}
X. D. Gu,  S. T.  Yau,  Computational conformal geometry. Advanced Lectures in Mathematics (ALM), 3. International Press, Somerville, MA; Higher Education Press, Beijing, 2008.



\bibitem{GH2001}
K. G\"{u}rlebeck,  A. Hommel,
{ On finite difference Dirac operators and their fundamental solutions},
Adv. Appl. Clifford Algebr.  {11}, 89-106 (2001).



\bibitem{GS2}
K. G\"{u}rlebeck,  W. Spr\"{o}ssig,   Quaternionic Analysis and Elliptic Boundary Value Problems. International Series of Numerical Mathematics, 89. Birkh\"{a}user Verlag, Basel, 1990.


      \bibitem{HR}
Q. H. Huo, G. B.   Ren,   Structure of octonionic Hilbert spaces with applications in the Parseval equality and Cayley-Dickson algebras. J. Math. Phys. 63 (2022), no. 4, Paper No. 042101, 24 pp.







\bibitem{Isaac}
 R. P. Isaacs,  A finite difference function theory. Univ. Nac. Tucum\'{a}n. Revista A. 2, (1941), 177-201.










\bibitem{KLL}    S. Karigiannis, N. C. Leung, J. D. Lotay,   Lectures and Surveys on $G_2$-manifolds and Related Topics.  Fields Institute Communications, 84. Springer, New York,  2020.

\bibitem{Kr}
B. Krause,   Discrete analogues in harmonic analysis—Bourgain, Stein, and beyond. Graduate Studies in Mathematics, 224. American Mathematical Society, Providence, RI (2022).


   
















    \bibitem{KLL2021}
    R. S. Krausshar, A. Legatiuk, and D. Legatiuk,
    {    Towards Discrete Octonionic Analysis},  International Conference on Differential Equations, Mathematical Modeling and Computational Algorithms. Cham: Springer International Publishing,  51-63 (2021).


\bibitem{KL}
R. S. Krausshar, D. Legatiuk,
Cauchy formulae and Hardy spaces in discrete octonionic analysis,
arXiv:2305.04121v1, 2023.



\bibitem{LPT2008}
X. Li, L. Peng, T. Qian,
 {Cauchy integrals on Lipschitz surfaces in octonionic space},
  J. Math. Anal. Appl, {343}(2): 763-777 (2008).

\bibitem{Me}
C. Mercat, Discrete Riemann surfaces and the Ising model, Comm. Math. Phys. 218:1 (2001),
177–216.





\bibitem{RZ2023}    G. B. Ren, Y. C. Zhang, Hartogs phenomena for discrete $k$-Cauchy-Fueter operators. J. Geom. Anal. 33 (2023), no. 8. 



        \bibitem{RZhao} G. B. Ren, X. Zhao,  The explicit twisted group algebra structure of the Cayley-Dickson algebra. Adv. Appl. Clifford Algebr. 33 (2023), no. 4, Paper No. 49, 14 pp.


    \bibitem{RZ2017}
G. B. Ren, Z. P.  Zhu,
 {A Convergence Relation Between Discrete
	and Continuous Regular Quaternionic
	Functions},  Adv. Appl. Clifford Algebr. \textbf{27}, 2017, no. 2, 1715-1740.

\bibitem{RZ2022}
G. B.  Ren, Z. P. Zhu,
 {Plemelj projections in discrete quaternionic analysis}. Math. Ann. \textbf{382}(2022), no.1-2, 975-1025.



\bibitem{S1954}
R. D. Schafer,
 {On the Algebras Formed by the Cayley-Dickson Process}, Am. J. Math.,  {76}(2), 435-446 (1954).




\bibitem{SW1979}
P. N. Shivakumar, R. Wong,  {Asymptotic expansion of multiple Fourier transforms}. SIAM J. Math. Anal. \textbf{10}(6), 1095-1104 (1979).




\bibitem{Sk}
M. Skopenkov,  The boundary value problem for discrete analytic functions. Adv. Math. 240, 61–87
(2013)



\bibitem{Smirnov}
S. Smirnov,   Conformal invariance in random cluster models. I. Holomorphic fermions in the Ising model. Ann. of Math. (2) 172 (2010), no. 2, 1435-1467.



\bibitem{Sm}
S. Smirnov, Discrete complex analysis and probability, in: Proc. Intern. Cong. Math. Hyderabad, India, 2010.




\bibitem{S}
R. J. Szabo,  An introduction to nonassociative physics,
arXiv:1903.05673, 2019








\bibitem{Th}
V. Thomée,  Discrete interior Schauder estimates for elliptic difference operators.
SIAM J. Numer. Anal. 5 (1968), 626–645.




\bibitem{Wo}
R. Wong,   Asymptotic approximations of integrals. Corrected reprint of the 1989 original. Classics in Applied Mathematics, 34. Society for Industrial and Applied Mathematics, Philadelphia, PA, 2001.



\bigskip\bigskip\bigskip
\end{thebibliography}
\end{document}